\documentclass[11pt]{article}
\usepackage{amsfonts, epsfig, amsmath, amssymb, color,amsthm}
\usepackage{textcomp}
\usepackage{paralist}

\usepackage[english]{babel}

\renewcommand {\epsilon}{\varepsilon}

\newcommand{\EE}{\mathbb{E}}

\newcommand{\MM}{\mathbb{M}}
\newcommand{\NN}{\mathbb{N}}

\newcommand{\PP}{\mathbb{P}}

\newcommand{\RR}{\mathbb{R}}

\newcommand{\TT}{\mathbb{T}}

\newcommand{\bB}{\mathcal{B}}
\newcommand{\cC}{\mathcal{C}}
\newcommand{\dD}{\mathcal{D}}

\newcommand{\fF}{\mathcal{F}}

\newcommand{\iI}{\mathcal{I}}

\newcommand{\mM}{\mathcal{M}}

\newcommand{\qQ}{\mathcal{Q}}

\newcommand{\tT}{\mathcal{T}}

\newcommand{\fB}{\mathfrak{B}}

\newcommand{\fs}{\mathfrak{s}}

\newcommand{\al}{\alpha}

\newcommand{\e}{\varepsilon}
\newcommand{\la}{\lambda}

\newcommand{\si}{\sigma}

\newcommand{\vp}{\varphi}

\newcommand{\pd}{\partial}

\newcommand{\ra}{\rightarrow}

\newcommand{\lra}{\longrightarrow}
\newcommand{\ti}{\widetilde}
\newcommand{\ha}{\widehat}

\newcommand{\ds}{\displaystyle}
\newcommand{\ind}{\mathbf{1}}

\newcommand{\lqq}{\leqslant}
\newcommand{\gqq}{\geqslant}

\newtheorem{thm}{Theorem}[section]

\newtheorem{prp}[thm]{Proposition}
\newtheorem{lem}[thm]{Lemma}

\theoremstyle{definition}

\newtheorem{exa}[thm]{Example}

\DeclareMathSymbol{\ophi}{\mathalpha}{letters}{"1E}

\renewcommand{\phi}{\varphi}

\newcommand{\be}{\begin{equation}}
\newcommand{\ee}{\end{equation}}
\newcommand{\ben}{\begin{equation*}}
\newcommand{\een}{\end{equation*}}

\newcommand{\ba}{\begin{equation}\begin{aligned}}
\newcommand{\ea}{\end{aligned}\end{equation}}

\DeclareMathOperator{\supp}{supp}

\DeclareMathOperator{\dist}{dist}

\newcommand{\bI}{\mathbf{1}} 
\newcommand{\bR}{\mathbb{R}}

\newfont{\cyrfnt}{wncyr10}
\def\J3{\cyrfnt{\rm \u{\cyrfnt I}}}
\def\j3{\cyrfnt{\rm \u{\cyrfnt i}}}

\allowdisplaybreaks[4]

\begin{document}
\title{Metastability of Morse--Smale dynamical systems\\
perturbed by heavy-tailed L\'evy type noise}

\date{\null}

\author{Michael H\"ogele\footnote{Institut f\"ur Mathematik, Universit\"at Potsdam,
Am Neuen Palais 10, 14465 Potsdam, Germany; hoegele@uni-potsdam.de}  \ 
and 
Ilya Pavlyukevich\footnote{Institut f\"ur Mathematik, Friedrich--Schiller--Universit\"at Jena, Ernst--Abbe--Platz 2, 
07743 Jena, Germany; ilya.pavlyukevich@uni-jena.de}
}

\maketitle

\begin{abstract}
We consider a general class of finite dimensional deterministic dynamical systems 
with finitely many local attractors $K^\iota$ 
each of which supports a unique ergodic probability measure $P^\iota$, which includes 
in particular the class of Morse--Smale systems in any finite dimension. 
The dynamical system is perturbed by a multiplicative non-Gaussian heavy-tailed 
L\'evy type noise of small intensity $\e>0$. Specifically we consider 
perturbations leading to a It\^o, Stratonovich and canonical 
(Marcus) stochastic differential equation. 
The respective asymptotic first exit time and location problem 
from each of the domains of attractions $D^\iota$ in case of inward pointing vector fields 
in the limit of $\e\searrow 0$ was solved by the authors in \cite{HoePav-14}.
We extend these results to domains with characteristic boundaries and 
show that the perturbed system exhibits a metastable behavior 
in the sense that there exits a unique $\e$-dependent 
time scale on which the random system converges to a continuous time Markov chain switching between 
the invariant measures $P^\iota$. As examples we consider $\alpha$-stable perturbations 
of the Duffing equation and a chemical system exhibiting a birhythmic behavior. 
\end{abstract}



\section{Introduction}

Consider a multivariate deterministic dissipative dynamical system 
given as the solution flow of a 
finite-dimensional ordinary differential equation $\dot u=f(u)$.
We assume that it has finitely many local attractors $K^\iota$,
each of which is contained in a 
domain of attraction $D^\iota$. 
By definition, for each initial condition in $D^\iota$ the trajectory never leaves $D^\iota$ and converges to $K^\iota$. 
We shall not impose specific conditions on the geometry of the attractors 
instead we assume that the time averages of the trajectories 
converge weakly to the unique invariant probability measure $P^\iota$ supported on $K^\iota$ as time tends to infinity. 
This convergence should be uniform w.r.t.\ the trajectory's initial condition over compact subsets 
of the domain $D^\iota$. 
Dynamical systems with finitely many stable fixed points $K^\iota= \{\mathfrak{s}^\iota\}$ 
or stable limit cycles belong to the 
evident examples of systems under consideration. 

The behavior of the system changes significantly in the presence of a perturbation 
by noise, however small its intensity $\e>0$ may be. 
In the generic situation, the perturbed solution relaxes from the 
initial position and remains --- usually for a very long time --- 
close to the attractor $K^\iota$ of the initial domain $D^{\iota}$.  
However with probability one, it exits from $D^\iota$ at some random time instant in an abrupt move 
and immediately enters another domain $D^j$, $j \neq \iota$, 
where the same performance starts anew.  
In this way, step by step and after possibly many repetitions 
the process visits all domains, not all of them of course 
with the same frequency and for an equally long period. 
In the literature, such a behavior of the trajectory is referred to as \textit{metastability}. 

In Galves et al.\ \cite[p.\ 1288]{GalvesOV-87}, the authors describe 
the metastable behavior of a deterministic dissipative dynamical system subject to 
small \textit{Gaussian} perturbations as follows: 
``A stochastic process with a unique stationary 
invariant measure, which [...] behaves for a very long time as if it were described 
by another ``stationary'' measure (metastable state), performing [...] 
an abrupt transition to the correct equilibrium. 
In order to detect this behavior, it is suggested [...] to look at the time averages along typical trajectories; 
we should see: apparent stability --- sharp transition --- stability.''

In any case, the transition times between different domains of attraction 
tends to infinity as the noise amplitude $\e$ goes to zero, however, 
the growth rate of the expected transition time 
as well as the probability to pass from $D^\iota$ to $D^j$ 
strongly depend on the nature of the noise and the properties of the underlying deterministic system.

In this article, we study the behavior of a dynamical system 
given as the solution flow of a rather generic finite-dimensional ordinary differential equation $\dot u=f(u)$
subject to a small noise perturbation by a multiplicative L\'evy type 
noise with a discontinuous, non-Gaussian heavy-tailed component. 
Since its dynamics will differ strongly from the case of Gaussian perturbations,
let us briefly discuss the underlying deterministic dynamical system and 
summarize the 
metastability results in the Gaussian case.

\subsection{Generic dynamical systems under consideration}

There is a large body of literature on the classification of deterministic dynamical systems 
and their stability properties, which we obviously cannot review here. 
Instead, we will restrict ourselves 
to the minimal necessary orientation of the reader about the systems we consider in this article.  
In the sequel we will mainly refer to the overview articles \cite{AV09,PP08}, 
introductory books \cite{HalKoc91,Te12}, and the extensive list of references therein. 

The class of dynamical systems we consider has finitely many well separated local attractors, 
with respective domains of attractions. We suppose that all trajectories starting in a compact set 
inside the a domain of attraction converge weakly and uniformly to a unique invariant probability measure concentrated 
on the local attractor.  
This invariant measure is assumed to be parametrized by the sojourn times of the dynamical system on the attractor. 
Since this class is not classical we briefly give a subsumption of its relation into well-known classes. 

The simplest class of examples are gradient systems, where $f$ is given as the gradient $-\nabla U$ 
of a smooth non-degenerate multi-well potential function $U\colon \RR^d \ra \RR$ with finitely many minima 
$\mathfrak{s^\iota}$, $\iota =1, \dots, \kappa$.
In this case, the local invariant measure is given as a unit point mass $P^\iota = \delta_{\fs^\iota}$. 

A finite-dimensional dynamical system is said to have the Morse--Smale property 
if the set of its non-wandering points consists of a union of finitely many periodic orbits (limit cycles), 
whose points are all hyperbolic and whose invariant manifolds meet transversally. 
For each of the non-trivial periodic stable orbits of the non-wandering sets of the 
Morse--Smale system, which parametrizes the corresponding limit cycle, say, $K^\iota$, 
we can define the invariant measure $P^\iota$ by 
\[
P^\iota(A) := \frac{1}{\tT_\iota} \int_{0}^{\tT_\iota} \ind_{A}(u(s;x))\, ds, 
\quad A\mbox{ Borel }, \quad u(0;x)=x \in K^\iota,\  u(t + \tT_\iota;x) = u(t;x). 
\]
In the Appendix it is shown that a Morse--Smale dynamical system in any dimension 
over a compact domain satisfies the required property 
that for all initial conditions uniformly bounded from the separating manifold, 
the time average of the trajectory converges weakly to $P^\iota$. 
In dimensions $1$ and $2$ Morse--Smale systems coincide with the class of structurally stable systems which are 
generic in the sense of being an open dense subset of all dynamical systems generated by $\cC^2$ vector fields, see \cite{Pe62,Pu67}.
It is known for a long time that in higher dimensions $d\gqq 3$, the Morse--Smale systems are a subclass of 
structurally stable systems but that the latter fail to be generic. 

We emphasize, however, that our assumptions are not restricted to the Morse--Smale systems, 
since we require only the existence of finitely many local attractors
satisfying the above mentioned statistical property on the convergence of the time averages. 

Finally we remark, that from a slightly different perspective
we can interpret 
the finitely many invariant measures $P^\iota$
as the ergodic components 
of the so-called \textit{Sinai--Bowen--Ruelle measure} (SRB-measure, for short), 
sometimes referred to as the \textit{physical measure}. For details we refer to the classical text~\cite{Bo75}
and for a more recent overview to \cite{Yo02}.

\subsection{The hierarchy of cycles and time scales in the generic Gaussian case} 

The small noise analysis and metastability results for randomly perturbed dynamical systems of the form $dX_t=f(X_t)dt+\e \, dW$, 
$W$ being a Brownian motion
(the noise term may be multiplicative as well) may be performed with the help of the large 
deviations theory by Freidlin and Wentzell
\cite{FreidlinW-98}. 
It is well known that with any $D^\iota$ that contains a unique point attractor 
$K^\iota=\{\fs^\iota\}$ we can associate a positive number $V_\iota$ such that
the expected exit time from $D^{\iota}$ is asymptotically proportional 
to $\exp(V_{\iota}/\e^{2})$ in the limit of $\e \searrow 0$.
This result is a version of what is known as Kramers' law \cite{Kramers-40} 
in the physics and chemistry literature. The constant $V_{\iota}$ can be interpreted 
as the height of the lowest ``mountain pass'' on the way from the attractor $\fs^\iota$ to the boundary $\partial D^\iota$
in the 
energy landscape given by the so-called \textit{quasi-potential} determined by the vector field $f$.
The same result would hold for an arbitrary attractors $K^\iota$ whose points are equivalent w.r.t.\ the quasi-potential,
that is do not require any additional work for transitions between them (for example like in the case of a limit cycle).

Further, for any two domains $D^i$ and $D^j$, $i\neq j$, 
there is a number $V_{ij}\gqq 0$ such that the 
expected transition time from $D^{i}$ 
to $D^{j}$ is asymptotically proportional 
to $\exp(V_{ij}/\e^{2})$.  
Note that in the generic case the constants $V_{ij}$ are different and the time scales 
$\exp(V_{ij}/\e^2)$ are thus exponentially separated. 
This naturally leads to the hierarchy of consecutive transitions of the random trajectory staring in
$D^\iota$, the so-called the hierarchy of cycles.

Indeed, starting in $D^\iota$, we determine the unique sequence of indices $j(0)=\iota$, $j(1),j(2),\dots$, defined such that 
$V_{j(k-1),j(k)}=\min_{j\neq j(k-1)} V_{j(k-1),j}$, $k\gqq 0$.
The sequence $\{j(k)\}$ is periodic with some period $p_1$ and the states $C(1)=\{\iota,j(1),\dots, j(p_1-1)\}$ 
constitute the cycle of the first rank. For $C(1)$ we can analogously define cycles of the higher orders, the last cycle
containing all the states $\{1,\dots,\kappa\}$. Each cycle $C$ contains the main state $K(C)$, that is the index of the attractor,
in the basin of which the random trajectory spends most of its time before leaving the set $\cup_{j\in C} D^j$.
For a detailed exposition we refer to Freidlin and Wentzell \cite{FreidlinW-98} or to a recent work by Cameron \cite{Ca13}. 

It is a distinguishing property of a system perturbed by a small Gaussian noise that the hierarchy of cycles, their main states and
the logarithmic rates of the associated exponentially large transition times are
not random and are determined by the vector field $f$ with the help of the quasipotential.

Various refinements and generalizations of these results include the proof of the convergence of a small noise 
diffusion $X$ in a double-well 
potential to a two-state Markov chain \cite{GalvesOV-87, KipnisN-85}, a connection between the metastability and the spectrum 
of the diffusion's generator \cite{BerGen-10,BovierEGK-04,BovierGK-05, Kolokoltsov-00, KolokoltsovM-96},
or the study of the infinite dimensional systems \cite{BerGen-12, Br91,Br96, FL82, Fr88}.

\subsection{The unique time scale and total communication of states in the generic regularly varying L\'evy case}

In this paper we treat a $d$-dimensional dynamical system $\dot u=f(u)$ perturbed by a (multiplicative) L\'evy
noise with heavy-tailed jumps, that is a process whose L\'evy measure possesses regularly varying tails
with the index $-\alpha<0$. As an example of such a perturbation one can have in mind 
$\alpha$-stable L\'evy noise, $\alpha\in(0,2)$.

To our best knowledge, the Markovian systems with heavy-tailed jumps were firstly studied by Go\-do\-van\-chuk \cite{Godovanchuk-82}. 
The asymptotics of the first exit times an metastability results in the one-dimensional 
setting of systems represented by SDEs driven by additive 
heavy-tailed L\'evy processes were obtained in \cite{ImkellerP-06,ImkellerP-06}. 
Further the theory was developed for multivariate systems with heavy-tail 
multiplicative noise in \cite{ImkPavSta-10, Pavlyukevich11}
and for a class of stochastic reaction--diffusion equations in \cite{DHI13}. 

The behavior of a dynamical system perturbed by heavy jumps differs qualitatively from the 
Gaussian case. First, the behavior becomes non-local, that is 
by a single jumps of an arbitrary big magnitude the system may change its state instantly. 
Second, the power law jumps determining the heavieness of the jumps also determines the unique time scale on which the
exits from domains $D^i$ and 
transitions between the domains $D^i$ and $D^j$ occur.

For simplicitiy let us sketch the case of 
a small \textit{additive} perturbation by a stable L\'evy process $\e Z$ with the jump measure $\nu(A)=\int_A \|z\|^{-d-\alpha}dz$ . 
Let $K^i = \{\fs^i\}$ be a stable point. 
In this situation, the first exit time from the domain $D^i$ has the mean value $Q_i/\e^\alpha$ with the prefactor
$Q_i=\int_{\mathbb R^d\backslash D^i} \|z-\fs^i\|^{-d-\alpha}dz$. In other words, 
the prefactor $Q_{i}$ measures the set of all jump increments of the noise, 
whose result is the exit 
from the domain $D^i$ at a single jump.
We refer to \cite{DHI13, HoePav-14, ImkellerP-06,ImkellerP-08} for detailed explanations. 

To describe transitions between the different domains of attraction we will see that in contrast to the Gaussian 
hierarchy of cycles, all mean transition times from the domain $D^{i}$ to $D^{j}$
are asymptotically equivalent to $Q_{ij}/ \e^\alpha$ in the limit of small $\e$ for 
$Q_{ij}=Q_i^{-1}\int_{D^j}\|z-\fs^i\|^{-d-\alpha}dz $. 
This means that the transition rates are not well separated for small $\e$. 
This \textit{generic} picture in the heavy-tailed framework may be associated
with the very \textit{degenerate} Gaussian case when all logarithmic rates $V_{ij}$ are identical
and the transition behavior is determined by the sub-exponential prefactors.
For a very precise asymptotics of these prefactors 
in the Gaussian setting we refer to Kolokoltsov \cite{KolokoltsovM-96,Kolokoltsov-00} and Bovier et al.\ \cite{BovierEGK-04}.

In \cite{HoePav-14}, we generalize the exit time results 
to underlying deterministic generic dynamical systems with non-point attractors. 
The stable state $\fs^i$ as a \textit{geometric} object appearing in the formulae for the mean transition times
has to be replaced by a \textit{statistical} 
quantity given as the ergodic invariant probability measure $P^i$ concentrated on the local 
attractor $K^i$ of the respective domain $D^i$. More precisely
we prove that a transition time between domains $D^i$ and $D^j$ 
asymptotically grows as $\ti Q_{ij}/ \e^\al$
with 
\[
\ti Q_{ij} = \frac{\int_{K^i} \int_{D^j}\|z-v\|^{-d-\alpha} dz P^i(dv)} {\int_{K^i} \int_{\mathbb{R}^d\backslash D^i}\|z-v\|^{-d-\alpha} dz P^i(dv)}. 
\]
The coefficient $\ti Q_{ij}$ weights the points on the attractor $K^i$ with respect to 
the corresponding ergodic invariant measure $P^i$.
For details we refer to the introduction of \cite {HoePav-14}. 

We see that generically the expected transition time 
between any two domains of attraction is proportional to $1/\e^\al$. 
Moreover it is shown that the respectively renormalized 
transition times are asymptotically exponentially distributed. 
Let us consider the perturbed path $X^\e$ on the time scale $t/\e^\al$. On this time scale we would 
expect that the process $X^\e(\frac{t}{\e^\alpha})$ spends most of the time in the domains of attraction $D^i$
exhibiting instantaneous single jump transitions from the vicinity of the attractor $K^i$ to the domain $D^j$.
Thus the first result of this paper will describe a Markov chain $m=(m_t)_{t\gqq 0}$ on the index set
 $\{1, \dots, \kappa\}$, 
which will specify the domain of attraction $D^i$ the process $X^\e_{\cdot /\e^\alpha}$ currently sojourns.
Roughly speaking, this allows us to determine the probability for the process $X^\e_{\cdot /\e^\alpha}$
to visit domains $D^{i_1},\dots, D^{i_n}$ at prescribed deterministic times $0<t_1<\cdots< t_n$, $n\gqq 1$.

In the second part, we prove a stronger result. Under the condition that $X^\e_{t /\e^\alpha}\in D^i$ for some $i\in\{1,\dots,\kappa\}$,
the process $X^\e$ is naturally located in the vicinity of the attractor $K^i$. We will determine the location of $X^\e$
at a slightly randomized observation time $(t+\si r_\e)\e^{-\al}$, $\sigma$ being an independent random variable uniformly distributed on 
$[-1,1]$ and $r_\e$ being an arbitrary rate characterizing 
the time measurement error such that $r_\e\to 0$ and $r_\e/\e^\alpha\to \infty$. We show
that in the limit $\e\to 0$, 
the location $X^\e_{(t+\si r_\e)\e^{-\al}}$ is distributed on the attractor $K^i$ according to the ergodic measure $P^i$,
whereas the attractor index $i=m_t$ is itself distributed with the law of the Markov chain $m$.
Essentially this means that within a given vanishing error bound 
on the time scale $t/\e^\al$ only the statistical aggregate of the behavior 
$X^\e$ can be perceived. 

We can make the intuition presented above rigorous for a general class of additive and multiplicative 
L\'evy noises with a regularly varying L\'evy measure. 
In particular, our main result covers perturbations in the sense of It\^o and Stratonovich, as well as in the 
sense of canonical (Marcus) equation, where jumps in general do not occur along straight lines, 
but follow the flow of the vector field which determines the multiplicative noise. 

In the physics and other natural sciences, 
Gaussian perturbations of dynamical systems with limit cycle attractors 
have been considered since quite some time, see e.g.\ Epele et al.\ \cite{EFSV85}, 
Moran and Goldbeter \cite{MorGol84}, Hill \textit{et al.} \cite{HLP09}, 
Kurrer and Schulten \cite{KurSch91}, Liu and Crawford \cite{LC98}, and Saet and Viviani \cite{SV87}. 
As an application of our main result we present two examples in detail: the Duffing equation with two point attractors
and a planar system from \cite{MorGol84} with two stable limit cycles which lie in one another.

\section{Object of study and main result}\label{sec: object of study}
\subsection{Deterministic dynamics}

We consider a globally Lipschitz continuous 
vector field $f\in \cC^2(\RR^d, \RR^d)$. 
It is well-known that this assumption is sufficient to establish the existence and 
uniqueness of the dynamical system, 
given as the solution flow $\vp$ of the autonomous ordinary differential equation
\begin{equation}\label{eq: ode}
\dot u = f(u), \qquad u(0;x) = x \in \RR^d, 
\end{equation}
where we denote by $\vp_t(x) := u(t;x)$. 
Note that the dynamical system can be prolonged to arbitrary negative times. 

We assume the following properties of $\phi$. 
\begin{enumerate}
 \item The set of non-wandering points of $\phi$ contains 
finitely many local attractors $K^\iota$, $\iota =1, \dots, \kappa$, $\kappa\gqq 1$, 
with corresponding open domains of attractions $D^\iota$. 
For definitions we refer to \cite{AV09} and \cite{HalKoc91}. 

\item All non-wandering points of $\phi$ are hyperbolic and the corresponding invariant manifolds meet transversally. 

 \item For any $R>0$ such that $\bigcup_{\iota} K^\iota \subset B_R(0)$, there 
exits a bounded, measurable, connected set $\iI_R\subset B_R(0)$ with smooth boundary, 
such that $f\big|_{\pd \iI_R}$ is uniformly inward pointing. 

\item For each local attractor $K^\iota$ there exists a unique probability measure $P^\iota$ 
supported on $K^\iota$, $\supp(P^\iota) = K^\iota$, such that 
such that for all non-negative, measurable and bounded functions 
$\psi\colon \RR^d \ra \RR$, any $R>0$ defined in 3,  
and all closed subsets $A$ contained in the interior of $D^\iota \cap \iI_R$ the limit
\begin{equation}\label{eq: ergodicity}
\lim_{t\ra \infty} \sup_{x\in A} \frac{1}{t} \int_0^t  \psi(\vp_s(x)) ds = \int_{K^\iota} \psi(v) P^\iota(dv) 
\end{equation}
holds true. 
\end{enumerate}

\subsection{The random perturbation}

On a filtered probability space $(\Omega, \fF , (\fF_t)_{t\gqq 0},\mathbb{P})$, satisfying 
the usual hypotheses in the sense of Protter \cite{Protter-04}, 
we consider a L\'evy process $Z = (Z_t)_{t\gqq 0}$ with values in $\RR^m$, $m\gqq 1$, 
and the characteristic function
\begin{equation*}
 \EE e^{i\langle u, Z_1 \rangle} =\exp\Big( -\frac{\langle Au,u\rangle}{2}  
+i\langle b,u\rangle +\int\Big( e^{i\langle u, z \rangle}-1-i\langle u,z\rangle\bI_{B_1(0)}(z)\Big) \nu(dz)\Big),\ u\in\bR^m,
\end{equation*}
where $A$ is a symmetric nonnegative definite $m\times m$ (covariance) matrix, $b\in \RR^m$, and $\nu$ 
a $\si$-finite measure on $\RR^m$ satisfying 
\(\nu(\{0\}) = 0 \) and $\int_{\RR^m} (1 \wedge \|y\|^2) \nu(dy) < \infty$. 
The measure $\nu$ is referred to as the L\'evy measure of $Z$, 
and $(A, \nu, b)$ is called the generating triplet of $Z$. 

 Let us denote by $N(dt, dz)$ the associated Poisson random measure with the intensity measure $dt \otimes \nu(dz)$ 
and the compensated Poisson random measure 
$\ti N(dt, dz) = N(dt, dz) - dt \nu(dz)$. Consequently, by the L\'evy--It\^o theorem (see e.g.\ Applebaum \cite[Chapter 2]{Applebaum-09}) 
the L\'evy process~$Z$ given above has the following a.s.\ path-wise additive decomposition  
\begin{equation}
\label{eq: Levy-Ito}
Z_t=A^{\frac{1}{2}} B_t+b t + \int_{(0, t]}\int_{0<\|z\|\lqq 1} z \ti N(ds, dz) + \int_{(0,t]}\int_{\|z\|> 1} z N(ds, dz),\qquad t\gqq 0,
\end{equation}
with $B = (B_t)_{t\gqq 0}$ being a standard Brownian motion in $\RR^m$. 
Furthermore, the random summands in (\ref{eq: Levy-Ito}) are independent. 
For further details on L\'evy processes we refer to Applebaum  \cite{Applebaum-09} and Sato \cite{Sato-99}.

The following assumption about the big jumps of $Z$ is crucial for our theory. 

\noindent \textbf{(S.1)} The L\'evy measure $\nu$ of the process $Z$ is 
\textit{regularly varying at $\infty$} with index $-\alpha$, $\alpha>0$. 
Let $h\colon (0,\infty)\to (0,\infty)$ denote the tail of $\nu$  
\begin{equation}
\label{def: h}
h(r):=\int_{\|y\|\gqq r }\nu(dy).
\end{equation}
We assume that there exist $\alpha>0$ and a non-trivial self-similar Radon measure $\mu$ on $\bar \bR^m\backslash\{0\}$
such that $\mu (\bar\bR^m\backslash \bR^m)=0$ and
  for any $a>0$ and any Borel set $A$ bounded away from the origin, $0\notin \overline{A}$, with $\mu(\partial A)=0$, 
the following limit holds true:
\begin{equation}
\label{eq: regular variation}
\mu(aA)=\lim_{r\ra \infty} \frac{\nu(raA)}{h(r)} = 
\frac{1}{a^\alpha} \lim_{r\ra \infty} \frac{\nu(rA)}{h(r)}= \frac{1}{a^\alpha} \mu(A).
\end{equation}
In particular, following \cite{BinghamGT-87} there exists a positive function $\ell$ slowly varying at infinity 
such that 
\begin{equation*}
h(r) = \frac{1}{r^{\alpha} \ell(r)}\qquad\mbox{ for all} \quad r>0.
\end{equation*}
The self-similarity property of the limit measure $\mu$ implies that $\mu$ assigns no mass to
spheres centered at the origin of $\bR^m$ and has no atoms.
For more information on multivariate heavy tails and regular variation we refer
the reader to Hult and Lindskog \cite{HultL-06-1} and Resnick \cite{Resnick-04}.
The following set of assumptions deals with the multiplicative perturbation of the dynamical system $u$ 
by the L\'evy process $Z$.

\noindent \textbf{(S.2)}  
\noindent Consider continuous maps $G\in \cC(\mathbb{R}^d\times \mathbb{R}^m, \mathbb{R}^d)$ 
and $F, H\colon  \RR^d \ra \RR^d$ and fix the notation
\[
a(x, y) := F(x) F(y)^* \qquad \mbox{ for }x, y \in \RR^d,
\]
where $F(y)^*$ is the transposed (row) vector of $F(y)$. 
We assume that for any $R>0$ 
there exists $L = L_R >0$ such that
$f$, $G$, $H$ and $F$ satisfy the following properties. 

\begin{enumerate}
\item \textbf{Local Lipschitz conditions:} For all $x, y \in \iI_R$ 
\begin{multline*}
\|f(x) -f(y)\|^2  + \|a(x,x) - 2 a(x, y) + a(y,y)\| + \|H(x) - H(y)\|^2\\
+ \|F(x) - F(y)\|^2 +  \int_{B_1(0)} \|G(x,z)-G(y,z)\|^2 \nu(dz) \lqq L^2 \|x-y\|^2.
\end{multline*}

\item \textbf{Local boundedness: } For all $x\in \iI_R$ 
\begin{align*}
\|f(x)\|^2  + \|a(x,x)\| + \|H(x)\|^2 + \|F(x)\|^2 +  \int_{B_1(0)} \|G(x,z)\|^2 \nu(dz) \lqq L^2 (1+ \|x\|^2).
\end{align*}

\item \textbf{Large jump coefficient:} 
For all $x, y \in \iI_R$ and $z\in \RR^m$
\begin{align*}
\|G(x,z) - G(y, z)\| \lqq L e^{L (\|z\| \wedge L)} \|x-y\|. 
\end{align*}

\item \textbf{Local bound for $G$ in small balls:} 
There exists $\delta'>0$ such that for $z\in B_{\delta'}(0)$ 
\begin{align*}
\sup_{x\in B_{\delta'}(K^\iota)} \|G(x, z)\| \lqq L. 
\end{align*}
\end{enumerate}

\label{S.4 G to 0 for small balls}

\begin{prp}
\noindent Let the assumptions \emph{\textbf{(S.2.1--3)}} be fulfilled.  
Then for any $\e, \delta\in (0,1)$, $R> 0$ and $x\in \iI_R$ the stochastic differential equation 
\begin{equation}
\label{eq: sde}
\begin{aligned}
X_{t,x}^\e &= x + \int_0^t f(X^\e_{s,x})\, ds 
+ \e \int_0^t H(X^\e_{s,x}) b \, ds + \e \int_0^t F(X^\e_{s, x}) \, d (A^{\frac{1}{2}}B_s)\\
&\qquad  + \int_0^t \int_{\|z\| \lqq 1} G(X^\e_{s-,x}, \e z) \ti N(ds, dz) +  \int_0^t \int_{\|z\| > 1} G(X^\e_{s-,x}, \e z) N(ds, dz)  
\end{aligned}
\end{equation}
has a unique strong solution $(X^\e_{t \wedge \TT,x})_{t\gqq 0}$ with c\`adl\`ag paths in $\RR^d$ 
which is a strong Markov process with respect to $(\fF_t)_{t\gqq 0}$, 
where 
\begin{align*}
\TT =\TT_x^R(\e)&:=\inf\{t\gqq 0\colon  X^\e_{t,x}\notin \iI_R\}. 
\end{align*}
is the first exit time from $\iI_R$.
\end{prp}
A proof can be found for instance in Ikeda and Watanabe \cite{IW89}, Theorem 9.1, 
or Chapter 6 in Applebaum \cite{Applebaum-09}. 
The multiplicative perturbations in the sense of It\^o, Fisk--Stratonovich or (canonical) Marcus equations
could be of special interest for applications.   
We refer the reader to Applebaum \cite{Applebaum-09}, Ikeda Watanabe \cite{IW89} and Protter \cite{Protter-04} 
for a general theory of stochastic integration in the It\^o and Fisk--Stratonovich sense 
and to Applebaum~\cite{Applebaum-09}, Kurtz et al.\ \cite{KurtzPP-95}
and Kunita \cite{Kunita-04} for a construction of the canonical Marcus equations. 
A brief comparison of these equations can be also found in Pavlyukevich \cite{Pavlyukevich11}. 

For example, assume that $Z$ is a pure jump L\'evy process with $A=0$, $b=0$, and
let $\Phi\colon \bR^d\to\bR^{d\times m}$ be a globally Lipschitz continuous function. 
Taking
\[
G(x,z):=x-\Phi(x)z
\]
we yields the It\^o SDE with the multiplicative noise
\ba
\label{eq:ito}
X_t&=x+\int_0^t f(X_s)dt+\e\int_0^t \Phi(X_{s-})dZ_s,\\
\ea
To obtain a canonical (Marcus) equation with the multiplicative noise
\ba
\label{eq:marcus}
X_t^\diamond &=x+\int_0^t f(X_s^\diamond)dt+\e\int_0^t \Phi(X_{s-}^\diamond)\diamond dZ_s.
\ea
we denote by $\psi^z(x) = y(1;x, z)$ the solution of the nonlinear ordinary differential equation
\begin{equation}
\label{eq:suppl}
\begin{cases}
 &\dot y(s) = \Phi(y(s)) z,\\
  & y(0) =x , \quad s\in [0,1].
\end{cases}
\end{equation}
and set
\[
G(x,z):=\psi^z(x).
\]
If $L$ is the Lipschitz constant of the matrix function $\Phi$ then the 
Gronwall lemma implies that
\[
\|G(x,z)-G(y,z)\| \lqq L e^{L\|z\|} \|x-y\|\qquad \forall x, y \in D, z\in \RR^m, 
\] 
what justifies the assumption (S.2.3). 

\subsection{The main result and examples   \label{subsec: the main results}}

\noindent For $x\in \bR^d$, $U\in \fB(\RR^d)$ with $x\notin U$ 
we denote the set of jump increments $z\in \RR^m$ 
which send $x$ into $U$ by
\begin{align}\label{def: event E}
E^{U}(x)& :=\{z\in \RR^m\colon x + G(x,z)\in U\}. 
\end{align}

\noindent We define the measure $Q^\iota$ on $\fB(\RR^d)$ assigning 
\begin{align}\label{def: measure Q}
Q^\iota(U) &:=\int_{K^\iota} \mu(E^{U}(y))~dP^\iota(y),
\end{align}
where $P^\iota$ is a measure on $K^\iota$ defined in (D.1) and $\mu$ is a regularly varying limiting jump measure 
appearing in
\eqref{eq: regular variation}. 
For $\e>0$ denote
\begin{align*}
\la_\e^\iota &:= \int_{K^\iota} \nu\Big(\frac{E^{(D^\iota)^c}(y)}{\e}\Big) dP^\iota(y) 
\quad \mbox{ and }\quad h_\e := h\Big(\frac{1}{\e}\Big).
\end{align*}
Then the equation \eqref{eq: regular variation} implies  
\begin{align*}
\lim_{\e\ra 0+} \frac{\la_\e^\iota}{h_\e} = Q^\iota((D^\iota)^c).  
\end{align*}
The main result of this article is the following metastability result. 

\begin{thm}\label{thm: metastability}
Let assumptions \emph{\textbf{(D.1)}} and \emph{\textbf{(S.1-2)}} be fulfilled and suppose that for all
$\iota=1,\dots,\kappa$,
\begin{equation} \label{eq: non-degenerate separatrix}
Q^\iota\Big(\RR^d \setminus \bigcup_{\ell=1}^\kappa D^\ell\Big) = 0.
\end{equation}
Then there exists a continuous-time Markov chain $m=(m_t)_{t\gqq 0}$ with values in the set 
$\{1, \dots, \kappa\} $
and a generator matrix 
\begin{align}
\label{eq:Q}
\qQ 
= \begin{pmatrix} 
 -Q^1\left(\left(D^1\right)^c\right)   & Q^1\left(D^2\right) & \dots & Q^1(D^{\kappa}) 
\\ 
\vdots&&&\vdots\\
Q^{\kappa} \left(D^1\right)  & \dots  & Q^{\kappa}(D^{\kappa-1}) & -Q^{\kappa}\left(\left(D^{\kappa}\right)^c\right)   
  \end{pmatrix}.
\end{align}
such that the following statements hold.
\begin{enumerate}
 \item Let $N\gqq 1$, $\iota_0,\dots, \iota_N \in \{1, \dots, \kappa\}$, $x\in D^{\iota_0}$, and
$0< s_1 < \dots < s_N$. Then 
\begin{align*}
\lim_{\e \ra 0+} \PP_x\Big(X^\e_{\frac{s_1}{h_\e}}\in D^{\iota_1}, \dots, X^\e_{\frac{s_N}{h_\e}}\in D^{\iota_N} \Big) = 
\PP_{\iota_0}(m_{s_1} = \iota_1, \dots, m_{s_N} = \iota_N).
\end{align*}
 \item Let $\si$ be a random variable which is uniformly distributed on $[-1, 1]$ and independent of $Z$. 
Let $r_\e\colon \RR_+\to\RR_+ $ be such that $r_\e \searrow 0$ and  $r_\e h_\e^{-1} \nearrow\infty$ as $\e \searrow 0$. 
Let $\psi \in \cC_b(\RR^{d}, \RR)$, $\iota\in \{1,\dots,\kappa\}$, and $0<s<t$. Then 
\begin{align*}
&\lim_{\e\ra 0+} 
\EE\Big[\psi\Big(X^\e_{\frac{t+\si r_\e}{h_\e} ,x}\Big)\Big| X^\e_{\frac{s}{h_\e},x} \in D^{\iota}\Big]
= \EE \Big[\int_{\RR^d} \psi(v) \, dP^{m_{t}}(v)\Big| m_{s} = \iota\Big]. \\
\end{align*}
\end{enumerate}
\end{thm} 

\begin{exa} 
We consider a damped low-friction Duffing equation
\ba
\label{eq:d}
\ddot{x}_t+\delta \dot x_t-U'(x_t)=0, \qquad \delta>0,
\ea
where $U(x)=\frac{x^4}{4}-\frac{x^2}{2}$ is a standard quartic potential.
We rewrite the equation \eqref{eq:d} as a system of two ODEs and perturb it by the multiplicative two-dimensional $\alpha$-stable 
L\'evy noise in the Marcus sense resulting in the 
two-dimensional SDE
\[
X^\e_t=x+\int_0^t f(X_s^\e)\, ds + \e \int_0^t G(X_s^\e)\diamond dZ_s,\\
\]
where
\begin{align*}
f(u)= 
\begin{pmatrix}
 u_2\\
-\delta u_2+ U'(u_1)
\end{pmatrix},\quad 
G(u)=\begin{pmatrix}
 0 & u_2\\
 u_1& 0
\end{pmatrix}
\end{align*}
The process $Z$ has the L\'evy measure $\nu(dz)=\frac{\alpha}{2\pi}\|z\|^{-2-\alpha}\bI(z\neq 0)dz$, where
we choose the normalization in such a way that
\[
h_\e=\frac{\alpha}{2\pi}\int_{\|z\|\gqq \frac{1}{\e}} \frac{dz }{\|z\|^{2+\alpha}}=\e^\alpha,\quad \e>0. 
\]
The unperturbed dynamical system $\dot u=f(u)$ has two stable point attractors $\fs_\pm=(\pm 1,0)$ with the domains of attraction 
$D_\pm$ separated by the
separatrix consisting of two branches which are particular
solutions of the ODE
\[
dy_\pm(t)=-f(y_\pm(t))\,dt
\]
with
$y_\pm(0)=0$ and $\dot y_\pm(0)=(1,\pm \lambda)$, with
\[
\lambda=\frac{\delta -\sqrt{\delta^2+4}}{2}.
\]
The form of the supplementary Marcus flow $\psi^z(x)$, see \eqref{eq:suppl}, is 
found explicitly. For for the attractors $x=\fs_\pm=(\pm 1,0)$ we get
\begin{equation*}
\psi^z(\fs_\pm)=
\begin{cases}
&\begin{pmatrix}
\pm\cosh\sqrt{z_1 z_2}\\
\pm\operatorname{sign} z_1 \sqrt{\frac{z_1}{z_2}} \sinh\sqrt{z_1 z_2} 
              \end{pmatrix},\quad z_1z_2> 0;\\
&\begin{pmatrix}
\pm\cos\sqrt{|z_1 z_2|}\\
\pm\operatorname{sign} z_1 \sqrt{\big|\frac{z_1}{z_2}\big|} \sin\sqrt{|z_1 z_2|} 
              \end{pmatrix}, \quad z_1z_2< 0;            
\\
&\begin{pmatrix}
\pm 1\\
\pm z_1
 \end{pmatrix}, \quad z_2= 0.
\end{cases}
\end{equation*}
We define the sets of jump increments which lead to a transition from $\fs_\pm$ to $D_\mp$ as
\[
E^{\pm}:=\{z\in\RR^2\colon \psi^z(\fs_\pm)\in D_\mp\}
\]
Then on the time scale $\frac{t}{\e^\alpha}$, 
the perturbed Duffing system $X^\e(\cdot /\e^\alpha)$ converges to a Markov chain $m(\cdot)$ in the sense of finite dimensional 
distributions where $m=(m(t))_{t\gqq 0}$ has the state space $\{\fs_-,\fs_+\}$ and the generator
\[
\qQ=\begin{pmatrix}
     -Q^-&Q^-\\
     Q^+&-Q^+
    \end{pmatrix}\quad \text{with}\quad 
Q^\pm:=\frac{\alpha}{2\pi}\int_{E^\pm}\frac{dz}{\|z-\fs_\pm\|^{2+\alpha}}. 
\]
\end{exa}

\begin{exa} 
In \cite{MorGol84},  Moran and Goldbeter considered a nonlinear model of a biochemical system with two oscillatory
domains which includes two variables: the substrate and product concentrations $u_1$ and $u_2$. Those time
evolution is governed by the equation $\dot u=f(u)$ which, for a particular choice of parameters, takes the form 
\ba
\label{eq:gm}
f(u)&=\begin{pmatrix}
      v+1.3\frac{ (u_2)^4}{K^4+(u_2)^4} -10\phi (u)\\
      10\phi (u) -0.06 u_2-1.3\frac{(u_2)^4}{10^4+(u_2)^4}
     \end{pmatrix},
     \quad  v=0.255,\\
\phi(u)&=\frac{u_1(1+u_1)(1+u_2)^2}{5\cdot 10^6+(1+u_1)^2(1+u_2)^2}.
\ea
The parameter $v\in\bR$ denotes the normalized input of substrate. It was shown in \cite{MorGol84} that
this system enjoys the property of \textit{birhythmicity}, that is the coexistence of two nested stable limit cycles, 
see Fig.~\ref{fig:fig-gm}(a). 
The inner and outer cycles have periods $\tT_\text{i}\approx 327$ and  $\tT_\text{o}\approx 338$ respectively.
Domains of attraction $D_\text{i}$ and $D_\text{o}$ are separated by an unstable cycle. 
Denote the parametrizations of the cycles by 
$\phi_\text{i}=(\phi_\text{i}^1(s),\phi_\text{i}^2(s))_{s\in[0,\tT_\text{i})}$ and 
$\phi_\text{o}=(\phi_\text{o}^1(s),\phi_\text{o}^2(s))_{s\in[0,\tT_\text{o})}$.

An addition of a certain quantity of the substrate, i.e.\ an instant increase of $v$ causes a 
switch between two stable oscillatory regimes.
Perturbations of the system \eqref{eq:gm} by additive Gaussian white noise were studied in \cite{LC98}. 

\begin{figure}    
\begin{center}
(a) \fbox{\includegraphics[height=4.3cm]{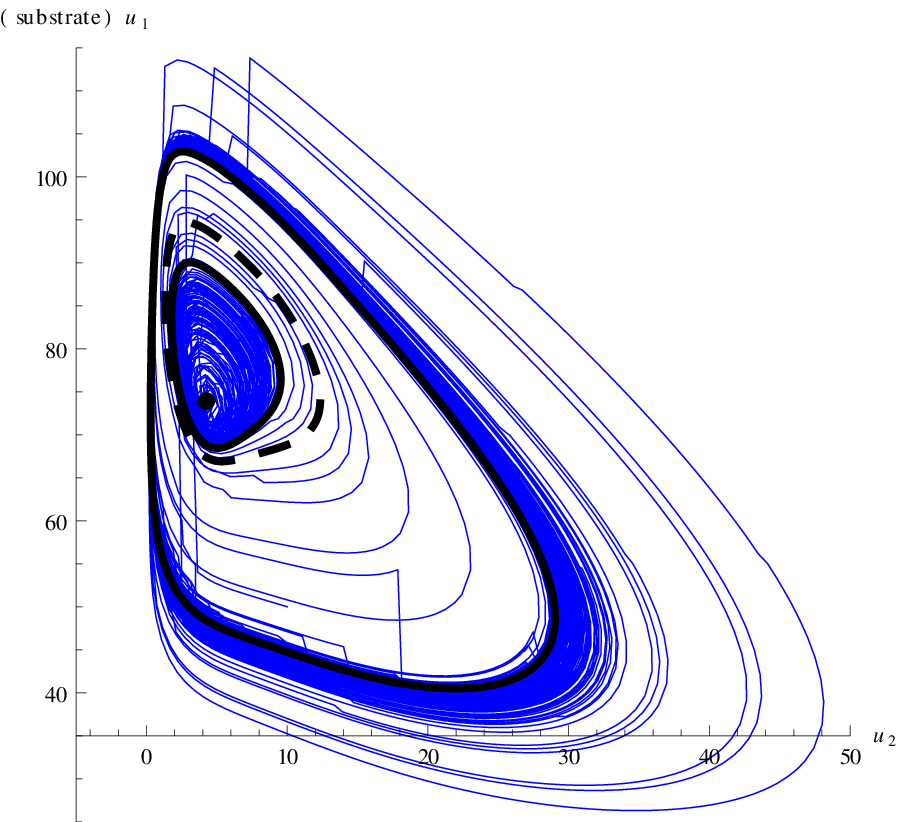}}
(b)
\fbox{\includegraphics[height=4.3cm]{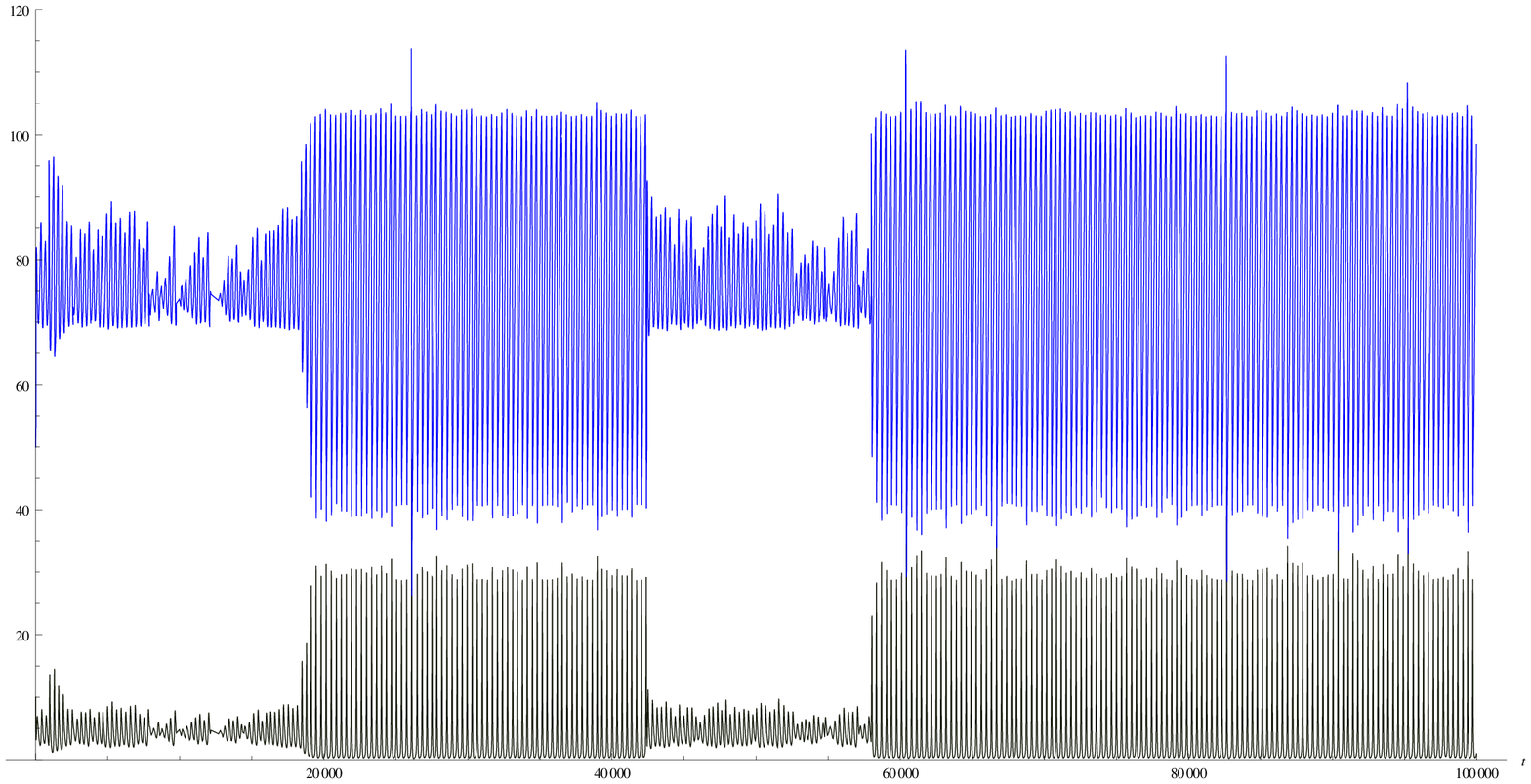}}
\end{center}
\caption{
(a) Coexisting nested stable cycles in the model of an autocatalytic reaction by Moran and Goldbeter \cite{MorGol84}. 
A heavy-tailed L\'evy perturbation of the substrate input
enables instant switchings between different ranges of 
substrate and product concentrations. 
(b) Random switching between periodic regimes of the substrate (blue) and the product (green) concentrations.
\label{fig:fig-gm}}
\end{figure} 

We perturb the parameter $v$ by a L\'evy process $Z$ which is a compound Poisson process with the Pareto jump measure
$\nu(dz)=\alpha z^{-1-\alpha}\bI(z\gqq 1)$, $\alpha>0$. We obtain the time scale rate
\[
h_\e=\alpha\int_{z>\frac{1}{\e}} \frac{dz}{z^{1+\alpha}}=\e^\alpha,\quad 0<\e<1.
\]
and the limiting self-similar Radon measure 
\[
\mu(dz)=\alpha\frac{\bI(z>0)}{z^{1+\alpha}}\, dz,\quad \alpha>0.
\]
On the time scale $\frac{t}{\e^\alpha}$, 
transitions between the cycles occur according to a law of the Markov chain $m$ on the state space 
$\{\text{i},\text{o}\}$ with the generator
\[
\qQ=\begin{pmatrix}
-Q^\text{i}&~Q^\text{i}\\
~Q^\text{o}&-Q^\text{o}
\end{pmatrix}
\]
where
\[ 
Q^\text{i}:=\frac{\alpha}{\tT_\text{i}}
\int_{0}^{\tT_{\text{i}}} \int_{1}^\infty \frac{\bI_{D_\text{o}}(z-\phi_\text{i}^1(s))}{|z-\phi_\text{i}^1(s)|^{1+\alpha}}\,dz\, ds
\quad\text{and}\quad
Q^\text{o}:=\frac{\alpha}{\tT_\text{o}}\int_{0}^{\tT_\text{o}} \int_1^\infty 
\frac{\bI_{D_\text{i}}(z-\phi_\text{o}^1(s))}{|z-\phi_\text{o}^1(s)|^{1+\alpha}}\,dz\,ds.
\]
It is clear from the phase portrait of the system $\dot u=f(u)$ that the area of the attraction basin $D_\text{i}$ 
is much smaller than the area of $D_\text{o}$ and thus 
$Q^\text{o}\ll Q^\text{i}$. Consequently, the system will spend 
most of the time in the vicinity of one of the stable cycles, preferably near the outer one, see Fig.~\ref{fig:fig-gm}(b). 
Any concrete measurement of 
concentrations will yield a random variable with the law 
$P^\text{i}$ or $P^\text{o}$ supported on the cycles, see \eqref{eq: ergodicity}. 
\end{exa}

\section{Proof}
\subsection{Preliminary results on the asymptotic first exit time }

The proof of the main Theorem \ref{thm: metastability} is based on a result about the first exit times of a perturbed system
from a domain $D$ around an attractor formulated in Theorem 2.1 in \cite{HoePav-14}. 
This result holds for deterministic vector fields $f$ which are inward pointing at the boundary of the bounded domain $D$. 
For general Morse--Smale systems this condition turn out to be too restrictive, since the boundary of domains of attraction 
$D^\iota$ is typically \textit{characteristic}, that is the vector field close to the separating manifold acts tangentially. 
Hence there are trajectories in the domain $D^\iota$ 
which may stay close to the separatrix for an uncontrollably long time  until they eventually converge to the attractor.
The proof of the Theorem 2.1 in \cite{HoePav-14} does not use precisely that the vector field is inward pointing, 
but rather the implication that a small reduction of the domain of attraction is still positively invariant 
and that all trajectories starting in the reduced domain are close to the attractor all together in time. 

Here we present another
construction of the reduced domains of attraction which is applicable to our setting.
It aims at avoiding the 
very slow dynamics near the characteristic boundary of the domain of attraction
and will not change the essential behavior of the stochastic system. 
An analogous construction had been carried out in Chapter 2.2.1 of \cite{DHI13}, 
for parabolic PDEs in the context of analysis of perturbed reaction-diffusion equations, 
with the additional difficulty that the latter do not have a backward flow. 

We fix $R>0$ and $\delta>0$ and consider the $\delta$-tube around the boundary $\pd D^\iota$ intersected with $ D^\iota \cap \iI_{R}$, 
namely
\[
\mM^{\iota, R}_\delta := \bigcup_{y\in \pd D^\iota} B_\delta(y) \cap D^\iota \cap \iI_{R}.
\]
 Then the set 
\[
\MM^{\iota, R}_\delta :=  \bigcup_{t\gqq 0} \phi_{-t}(\mM^{\iota, R}_\delta)
\]
denotes all initial values $x$ such that for some time $t\gqq 0$ the forward flow $\phi_t(x)$ enters $\mM^{\iota, R}_\delta$. 
We define the flow-adapted reduced domain of attraction 
\[
D^{\iota, R}_\delta := (D^\iota \cap \iI_R) \setminus \MM^{\iota, R}_{\delta}.
\]
For $\delta'>0$, iterating this procedure by replacing $D^{\iota, R}$ by $D^{\iota, R}_\delta$ and obtain further reductions  
\begin{align*}
\mM^{\iota, R}_{\delta, \delta'} &:= \bigcup_{y\in \pd D^{\iota, R}_{\delta}} B_{\delta'}(y) \cap D^{\iota, R}_{\delta},\\
\MM^{\iota, R}_{\delta, \delta'} &:=  \bigcup_{t\gqq 0} \phi_{-t}(\mM^{\iota, R}_{\delta, \delta'}),\\
D^{\iota, R}_{\delta, \delta'} &:= (D^\iota \cap \iI_R) \setminus \MM^{\iota, R}_{\delta, \delta'}.
\end{align*}

The reduced domains  $D^{\iota, R}_\delta$ and $D^{\iota, R}_{\delta, \delta'}$ enjoy the following important properties. 

\begin{lem}\label{lem: reduced domains}
Denote
\[
\delta_0 := \frac12\min_{1\lqq \iota\lqq\kappa }  \dist\Big(K^\iota, \ds\bigcup_{\iota=1}^\kappa \pd D^\iota\Big),\quad 
\text{and}\quad 
\ds R_0 := \inf\Big\{r>0\colon \bigcup_{\iota=1}^\kappa K^\iota \subset B_r(0)\Big\},
\]
and let $\iota \in \{1, \dots, \kappa\}$ be fixed.
\begin{enumerate}
 \item If $0 < \delta < \delta_0$ and $R>R_0$,
 then $\phi_t(D^{\iota, R}_{\delta}) \subset D^{\iota, R}_\delta$ for all $t\gqq 0$. 
 \item 
  If $0 < \delta < \delta_0$, $R>R_0$,
 and additionally $0< \gamma < \delta_0$, then there is $T^* = T^*_{\delta, R, \gamma}>0$ 
such that for all $x\in D^{\iota, R}_{\delta}$ and $t\gqq T^*$
\[
u(t;x) \in B_{\gamma}(K^\iota).
\]
This property corresponds to Remark 2.1 in \cite{HoePav-14}. 

\item If $0< \delta < \delta'< \delta_0$ and $R>R_0$, then $D^{\iota, R}_{\delta'} \subset D^{\iota, R}_{\delta}$.
 \item If $\delta, \delta'>0$ such that $\delta+ \delta'<\delta_0$ and $R>R_0$, then 
$\phi_t(D^{\iota, R}_{\delta, \delta'}) \subset D^{\iota, R}_{\delta, \delta'}$ for all $t\gqq 0$. 
 \item If $\delta, \delta', \delta''>0$ with $\delta'< \delta''$ and $\delta + \delta''< \delta_0$, then 
$D^{\iota, R}_{\delta, \delta'} \subset D^{\iota, R}_{\delta, \delta''}$.
 \item We have 
\[
\bigcup_{\substack{\delta, \delta'>0\\ \delta+\delta'<\delta_0}} D^{\iota, R}_{\delta, \delta'} = D^{\iota} \cap \iI_R.
\]
\end{enumerate}
\end{lem}
\noindent The proof of the Lemma is rather straightforward and postponed to the Appendix. 

Under an appropriate choice of parameters $R$, $\delta$, $\delta'$, $x\in D^{\iota, R}_{\delta, \delta'}$, $\e>0$ 
and $\iota \in \{0, \dots, \kappa\}$ we define the time  
\[
\TT_x^{\iota, R}(\e) := \inf\{t>0\colon X^\e_{t,x} \notin D^{\iota, R}_{\delta}\}.
\]
The next Theorem \ref{thm: first exit times}
is based on the Theorem 2.1 in \cite{HoePav-14} and deals with the behavior of $\TT_x^{\iota, R}(\e)$ in the limit of small $\e$. 
We will use the following version of Theorem 2.1 in \cite{HoePav-14} slightly adapted to our setting.

\begin{thm}[The exit problem of $X^\e$]\label{thm: first exit times}
Let Hypotheses \emph{\textbf{(D.1)}} and \emph{\textbf{(S.1-2)}} be fulfilled. Choose $R>R_0$, 
$\iota \in \{1, \dots, \kappa\}$ and $\delta, \delta'>0$ with $\delta+\delta'< \delta_0$. 
If $Q^\iota(\partial D^{\iota, R}_\delta) = 0$ and $Q^\iota((D^{\iota, R}_\delta)^c) >0$, then we have 
for any $\theta>0$
and $U\in \mathfrak{B}(\RR^d)$ 
satisfying $Q(\partial U) =0$ that  
\begin{align}\label{eq: first exit convergence}
\lim_{\e\to 0} \sup_{y\in D^{\iota, R}_{\delta, \delta'}}
\Big|\EE_y\Big[ e^{-\theta Q((D^{\iota, R}_{\delta})^c) h_\e  \TT^{\iota, R}(\e)} 
\ind\{X^\e_{\TT^{\iota, R}(\e)}\in U\}\Big]-
\frac{1}{1+ \theta} \frac{Q\big(U \cap (D^{\iota, R}_{\delta})^c\big)}{Q\big((D^{\iota, R}_\delta)^c\big)} \Big|=0.
\end{align}
\end{thm} 
\noindent This result implies that under the previous assumptions the first exit times and the first exit location behave as 
\begin{align*}
&h_\e\; Q^\iota((D^{\iota, R}_\delta)^c)\; \TT_x^{\iota, R}(\e) \stackrel{d}{\lra} \mbox{EXP}(1),\\
&\PP_x\Big(X_{\TT^{\iota, R}(\e)}^\e \in U\Big) 
\to \frac{Q^\iota\big(U\cap (D^{\iota, R}_\delta)^c\big)}{Q^\iota\big((D^{\iota, R}_\delta)^c\big)},
\end{align*}
in the limit $\e\to 0$,
where the convergence is uniform over all initial values $x \in D^{\iota, R}_{\delta, \delta'}$. These results 
allow the construction of a jump process, which converges weakly to an approximating continuous time Markov chain
$m$ with the generator \eqref{eq:Q}. 

\subsection{Proof of Theorem \ref{thm: metastability}}

Fix the error constant $\delta^*>0$. In the first step we fix the parameters $R$, $\delta$ and $\delta'$ accordingly 
and construct an approximating Markov chain.  

\paragraph{1. Approximating Markov chains.} 
The limiting measure $\mu$ of the regularly varying L\'evy measure $\nu$ 
given in \eqref{eq: Levy-Ito} is a Radon measure. 
We recall the definition of $R_0$ in Lemma \ref{lem: reduced domains}. 
We may fix a radius $R>R_0$, depending only on $\delta^*$, 
such that 
\[
\max_{\iota = 1, \dots, \kappa} Q^\iota(\iI_R^c) < \frac{\delta^*}{2}.  
\]
In addition, by compactness of $\iI_R$ we may fix one after the other, $\delta>0$ and $\delta'>0$, 
with $\delta+\delta'< \delta_0 = \frac12\min_{\iota =1, \dots, \kappa} \min\{\dist(K^\iota, \pd D^\iota), \dist(x, \dD^\iota)\}$
such that 
\[
\max_{\iota =1, \dots, \kappa} 
Q^\iota\Big(\iI_R \setminus \bigcup_{\ell =1}^{\kappa} D^{\ell, R}_{\delta, \delta'}\Big) < \frac{\delta^*}{2}, 
\]
where $\delta = \delta(R, \delta^*)$ and $\delta' =\delta'(R, \delta^*, \delta)$. 
Combining the previous two inequalities we obtain that
\begin{equation}\label{eq:delta star}
\max_{\iota =1, \dots, \kappa} Q^\iota\Big(\RR^d \backslash \bigcup_{\ell =1}^\kappa D^{\ell, R}_{\delta, \delta'}\Big) < \delta^*.
\end{equation}
We lighten the notation.  
For $\delta^*>0$ and the dependent parameters 
$R$, $\delta$, and $\delta'$ fixed we write shorthand 
$\ha D^\iota = D^{\iota, R}_{\delta}$ and $\ti D^\iota = D^{\iota, R}_{\delta, \delta'}$. 
Furthermore we use $A^c := \RR^d \setminus A$ for any $A \subset \RR^d$. 

Denote by $m^{\delta^*}= (m^{\delta^*})_{t\gqq 0}$ a continuous time 
Markov chain with values in the set of indices $\{1, \dots, \kappa\}\cup\{0\}$
enlarged by the absorbing cemetery state $0$ with the generator $\qQ^{\delta^*}$ given by 
\begin{align*}
 \qQ^{\delta^*} :=
&\left(
\begin{array}{ccccc} 
-Q^1\left(\big(\ha D^{1}\big)^c\right)  
& Q^1\left(\ti D^{2}\right) & \dots & Q^1\left(\ti D^{\kappa}\right) 
& \displaystyle Q^1\Big(\big(\ha D^1 \cup \bigcup_{\iota = 2}^\kappa \ti D^{\iota} \big)^c\Big)  \\ 
\vdots&&&\vdots \\
Q^{\kappa} \left(\ti D^{1}\right)  & \dots  & Q^{\kappa}\left(\ti D^{\kappa-1}\right) 
& -Q^{\kappa}\Big(\big(\ha D^{\iota_\kappa}\big)^c\Big) & \displaystyle
Q^{\kappa}\Big(\big(\ha D^\kappa \cup \bigcup_{\iota = 1}^{\kappa-1} \ti D^{\iota}\big)^c\Big) \\[2mm]
0& 0 &\dots & 0 & 0
\end{array} \right).
\end{align*}
For $\qQ$ defined in (\ref{eq:Q}) we construct the matrix
\[
\qQ^0 := \begin{pmatrix} \qQ & 0 \\ 0 & 0 \end{pmatrix} 
\]
and denote by $m^0$ a continuous-time Markov chain on the state space $\{1, \dots, \kappa\} \cup \{0\}$ 
with the generator $\qQ^0$. 
As a consequence of (\ref{eq:delta star}) we have 
\[
\max_{i, j} |\qQ^{\delta^*}(i,j) - \qQ^0(i,j)| < \delta^*. 
\]
This implies that $m^{\delta^*} \ra m^0$ as $\delta^* \searrow 0$ in the sense of finite dimensional distributions. 
Note that the transition rate to the cemetery state $0$ tends to $0$ as $\delta^*\searrow 0$ due to \eqref{eq:delta star}.

\paragraph{2. Transition probabilities.}
Let $N\gqq 1$, $\iota_0,\dots, \iota_N \in \{1, \dots, \kappa\}\cup \{0\}$, $x\in \ti D^{\iota}$ and $0< s_1 < \dots < s_N$. 
Let us show that
\begin{align}
\label{eq: delta approx fddX}
\lim_{\e \ra 0+} \PP_x\Big(X^\e_{\frac{s_1}{h_\e}}\in \ti D^{\iota_1}, 
\dots, X^\e_{\frac{s_N}{h_\e}}\in \ti D^{\iota_N} \Big) = 
\PP_{\iota_0}(m^{\delta^*}_{s_1} = \iota_1, \dots, m^{\delta^*}_{s_N} = \iota_N).
\end{align}
Since $0$ is an absorbing state, we can restrict ourselves to the states $\{1, \dots, \kappa\}$. 
We first construct an approximating jump process with the help of Theorem \ref{thm: first exit times} 
and define recursively the arrival times $\{T_n^\e\}_{n\gqq 0}$ 
and the random states $\{S_n^\e\}_{n\gqq 0}$ taking values in $\{0, \dots, \kappa\} \cup \{0\}$. 
We fix the initial time and state  
\begin{align*}
&T_0^\e := 0,\qquad S_0^\e:= \sum_{\ell = 1}^\kappa \ell\cdot \ind_{\ti D^{\ell}}(x).
\end{align*}
For $n\in \NN$ we set 
\begin{align*}
&T_{n+1}^\e := 
\begin{cases} \displaystyle
\inf\Big\{t>T_n^\e\colon X^\e_{t, x}\in \bigcup_{\substack{\ell =1\\ \ell \neq S_n}}^k \ti D^{\ell} \Big\}, 
& \ds \mbox{ if }S_n^\e \in \bigcup_{\ell=1}^\kappa \ti D^{\ell}, \\
\infty, & \ds\mbox{ if }S_n^\e \notin \bigcup_{\ell=1}^\kappa \ti D^{\ell},
\end{cases}\\
&S_{n+1}^\e := 
\begin{cases} 
\ds\sum_{\ell = 1}^\kappa \ell\cdot \ind_{\ti D^{\ell}}(X^\e_{T_{n+1}, x}),  &\ds \mbox{ if }T_{n+1}^\e < \infty, \\
\ds0, &\ds \mbox{ if } T_{n+1}^\e =\infty. 
\end{cases}
\end{align*}
We define the approximating jump process 
\begin{align*}
M^{\e, \delta^*}_t := \sum_{n=0}^\infty S_n^\e\cdot  \ind\big\{T_n^\e\lqq \frac{t}{h_\e} < T_{n+1}^\e\big\}.
\end{align*}
The convergence in \eqref{eq: delta approx fddX} can be expressed conveniently in terms of $M^{\e, \delta^*}$ 
as follows 
\begin{align}\label{eq: delta approx fdd}
\lim_{\e \ra 0+} \PP_x\Big(M^{\e, \delta^*}_{\frac{s_1}{h_\e}}=\iota_1, 
\dots, M^{\e, \delta^*}_{\frac{s_N}{h_\e}}=\iota_N \Big) = 
\PP_{\iota_0}(m^{\delta^*}_{s_1} = \iota_1, \dots, m^{\delta^*}_{s_N} = \iota_N).
\end{align}
Following for instance Lemma 2.12 and Lemma 2.13 in Xia \cite{Xia-92}, the convergence 
\[
M^{\e, \delta^*} \to m^{\delta^*} \quad \mbox{ as }\e\to 0 
\]
in the sense of finite dimensional distributions
it is equivalent to convergence 
\begin{align*}
(T_{k}^\e, S_k^\e)_{0\lqq k\lqq  n} \stackrel{d}{\lra} (T_k, S_k)_{0\lqq k\lqq  n}, 
\end{align*}
for any $n\in \NN$, where $T_k$ is the $k$-th arrival time for the Markov chain $m^{\delta^*}$ and 
$S_k = m^{\delta^*}_{\tau_k}$. 
This is equivalent to the following statement. 
For indices $\iota_0, \iota_1, \dots \iota_n \in \{1, \dots, \kappa\}$, 
\marginpar{can $\iota_k=0$?}
with $\iota_{k}\neq \iota_{k+1}$, $k\in \{0, \dots, n-1\}$, $u_1,\dots, u_n\gqq 0$, 
and an initial value $x \in \ti D^{\iota_0}$ we have 
\begin{align}
\EE_x\Big[e^{-u_1 T_1^\e - \dots - u_n (T_n^\e-T_{n-1}^\e)} \cdot
&\ind\{S_1^\e = \iota_1, \dots, S_n^\e = \iota_n\}\Big]\nonumber\\
&\qquad \stackrel{\e\to 0}{\lra} 
\prod_{j=0}^{n-1}
\frac{Q^{\iota_j}\big((\ha D^{\iota_j})^c\big)}{Q^{\iota_j}\big((\ha D^{\iota_j})^c\big)+u_{j+1}} 
\cdot \frac{Q^{\iota_j}\big(\ti D^{\iota_{j+1}}\big)}{Q^{\iota_j}\big((\ha D^{\iota_j})^c\big)}.
\label{eq: finite exit time results}
\end{align}
This implies the desired convergence of finite dimensional distributions \eqref{eq: delta approx fdd}.
To prove the convergence in \eqref{eq: finite exit time results} we use the strong Markov property of $X^{\e}$ for the 
following recursive estimate 
\begin{align*}
&\EE_x\Big[e^{-u_1 T_1^\e - \dots - u_n (T_n^\e-T_{n-1}^\e)}\cdot 
\ind\{S_1^\e = \iota_1, \dots, S_n^\e = \iota_n\}\Big]\\
&= \EE_x\Big[\EE\Big[e^{-u_1 T_1^\e - \dots - u_n (T_n^\e-T_{n-1}^\e)}\cdot 
\ind\{S_1^\e = \iota_1, \dots, S_n^\e = \iota_n\}\Big|\fF_{T_1^\e}\Big]\Big]\\
&= \EE_x\Big[e^{-u_1 T_1^\e}\cdot \ind\{S_1^\e = \iota_1\} \EE\Big[e^{- u_2 (T_2^\e-T_1^\e) - \dots - u_n (T_n^\e-T_{n-1}^\e)} 
\cdot \ind\{S_1^\e = \iota_2, \dots, S_n^\e = \iota_n\}\Big|\fF_{T_1^\e}\Big]\Big]\\
&= \EE_x\Big[e^{-u_1 T_1^\e} \cdot \ind\{X^\e_{T_1^\e} \in \ti D^{\iota_1}\}\; 
\EE_{X^\e_{T_1^\e}}\Big[e^{- u_2 (T_2^\e-T_1^\e) - \dots - u_n (T_n^\e-T_{n-1}^\e)}\cdot 
\ind\{S_1^\e = \iota_2, \dots, S_n^\e = \iota_n\}\Big]\Big]\\
&\lqq \EE_x\Big[e^{-u_1 T_1^\e}\cdot  \ind\{X^\e_{T_1^\e} \in \ti D^{\iota_1}\}\Big] 
\sup_{y\in \ti D^{\iota_1}} 
\EE_{y}\Big[e^{- u_2 T_1^\e - \dots - u_n (T_{n-1}^\e-T_{n-2}^\e)} \cdot \ind\{S_1^\e = \iota_2, \dots, S_n^\e = \iota_n\}\Big].
\end{align*}
We iterate the preceding argument $n-2$ times and obtain the estimate
\begin{align*}
&\EE_x\Big[e^{-u_1 T_1^\e - \dots - u_n (T_n^\e-T_{n-1}^\e)}\cdot 
\ind\{S_1^\e = \iota_1, \dots, S_n^\e = \iota_n\}\Big]\\
&\lqq \EE_x\Big[e^{-u_1 T_1^\e}\cdot  \ind\{X^\e_{T_1^\e} \in \ti D^{\iota_1}\}\Big] 
\prod_{\ell=1}^{n-1} \sup_{y\in \ti D^{\iota_\ell}}\EE_y\Big[e^{-u_\ell T_1^\e} \cdot 
\ind\{X^\e_{T_1^\e} \in \ti D^{\iota_{\ell+1}}\}\Big].
\end{align*}
The same reasoning holds true for the estimate from below 
if we change -mutatis mutandis- the supremum by the infimum. 
The limit \eqref{eq: first exit convergence} in Theorem \ref{thm: first exit times} states that 
\begin{align*}
\sup_{y\in \ti D^{\iota_\ell}}\EE_y\Big[e^{-u_\ell T_1^\e} \cdot \ind\{X^\e_{T_1^\e} \in \ti D^{\iota_{\ell+1}}\}\Big]
\stackrel{\e \to 0}{\lra} \frac{Q^{\iota_j}\big((\ha D^{\iota_j})^c\big)}{Q^{\iota_j}\big((\ha D^{\iota_j})^c\big)+u_{j+1}} 
\cdot \frac{Q^{\iota_j}(\ti D^{\iota_{j+1}})}{Q^{\iota_j}\big((\ha D^{\iota_j})^c\big)}.
\end{align*}
This shows the desired convergence in (\ref{eq: finite exit time results}) and finishes the proof of \eqref{eq: delta approx fdd}. 
Statement 1. of Theorem~\ref{thm: metastability} is proved. 

\paragraph{3. Location of $X^\e$ on the attractor.} We prove the second statement of the Theorem~\ref{thm: metastability}.
Since $X^\e$ is a strong Markov process, 
it is enough to prove the result for $s=0$ and $x\in \ti D^{\iota}$, namely that
\begin{align*}
&\lim_{\e\ra 0+} 
\EE_x\Big[ \psi(X^\e_{\frac{t}{h_\e},x})\Big]
= \EE_\iota\Big[\int_{\RR^d} \psi(v) \, dP^{m^\delta_{t}}(v)\Big].
\end{align*}
Indeed, the Markov property of $X^\e$ yields 
\begin{align}
\EE_x\Big[ \psi(X^\e_{\frac{t+\si r_\e}{h_\e}}) \Big]
&= \EE_x\Big[\frac{h_\e}{2 r_\e} \int_{-\frac{r_\e}{h_\e}}^{\frac{r_\e}{h_\e}} 
\psi(X^\e_{\frac{t+s}{h_\e}}) ds\Big]\nonumber\\
&= \EE_x\Big[\EE\Big[\frac{h_\e}{2 r_\e} \int_{-\frac{r_\e}{h_\e}}^{\frac{r_\e}{h_\e}} 
\psi(X^\e_{\frac{t+s}{h_\e}})\,  ds\Big| \fF_{\frac{t-r_\e}{h_\e}}\Big]\Big]\nonumber\\
&= \EE_x\Big[\EE_{X^{\e}_{\frac{t-r_\e}{h_\e}}}\Big[\frac{h_\e}{2r_\e} \int_{0}^{\frac{2r_\e}{h_\e}} 
\psi(X^\e_{s})\, ds\Big]\Big]\nonumber\\
&\lqq \sum_{\iota =1}^\kappa \EE_x\Big[\EE_{X^{\e}_{\frac{t-r_\e}{h_\e}}}\Big[\frac{h_\e}{2r_\e} 
\int_{0}^{\frac{2r_\e}{h_\e}}\psi(X^\e_{s})\, ds\Big] 
\cdot  \ind_{\ti D^\iota}(X^\e_{\frac{t-r_\e}{h_\e}}) \Big] + \|\psi\|_\infty \delta^*\nonumber\\
&\lqq \sum_{\iota =1}^\kappa \sup_{y\in \ti D^\iota} 
\EE_{y}\Big[\frac{h_\e}{2r_\e} \int_{0}^{\frac{2r_\e}{h_\e}} \psi(X^\e_{s})\, ds\Big]\cdot 
\PP_x\Big(X^\e_{\frac{t-r_\e}{h_\e}} \in \ti D^\iota\Big) + \|\psi\|_\infty \delta^*.
\label{eq: Markov estimate}
\end{align}

We treat the two factors of the summands separately. 

\begin{lem}\label{lem: 1}
Let $\delta^*$, $\delta'$, $\delta>0$ and $R>R_0$ be chosen as above. 
If $0< \gamma < \delta_0$, $\psi \in \cC_b(\RR^d, \RR)$ and $\iota \in \{1, \dots, \kappa\}$ 
then there is a constant $\e_0 >0$ such that for any $\e\in (0, \e_0]$  
\begin{align*}
\sup_{y\in D^{\iota, R}_{\delta, \delta'}} 
\Big|\EE_y\Big[\frac{h_\e}{2r_\e} \int_0^{\frac{2r_\e}{h_\e}} \psi(X^\e_s)\, ds\Big] 
- \int_{K^\iota} \psi(v)\, dP^\iota(v)\Big| \lqq \gamma.
\end{align*}
\end{lem}

\begin{proof} Fix $0 < \gamma < \delta_0$. For convenience we return to the abbreviation $\ti D^\iota$.
The local ergodicity condition \eqref{eq: ergodicity} of the deterministic dynamical system 
ensures the existence of a constant $\tT^*>0$ such that for all $\tT\gqq \tT^*$ 
\[
\max_{\iota \in \{1,\dots, \kappa\}} 
\sup_{y \in \overline{\ti D^{\iota}}} 
\Big|\frac{1}{\tT} \int_0^\tT \psi(\phi_s(y))\, ds - \int_{K^\iota} \psi(v)\, dP^\iota(v)\Big|< \frac{\gamma}{3}.
\]
According to Lemma \ref{lem: reduced domains}.1.(b) there is a constant $T^*>0$ depending on $R, \delta$ and $\gamma$
which ensures that for all $y \in \ti D^\iota$ and $t\gqq T^*$ 
\[
\dist(\phi_t(y), \pd \ti D^\iota)> \delta_0.  
\]
We choose $\tT^*\gqq T^*$ without loss of generality. 
Denote by $\ell_\e := \lfloor 2r_\e/h_\e \tT^*\rfloor$ the maximal number of times how often $\tT^*$ 
fits into $2r_\e/h_\e$. Then $\tT_\e := 2r_\e/h_\e \ell_\e$ 
satisfies $\tT^*\lqq \tT_\e < 2\tT^*$ for any $\e>0$. 
It is well-known that for any $\rho\in (0,1)$ and $\e>0$ the random variable 
\[\tau := \inf\{t>0\colon  |\Delta_t Z|>\e^{-\rho}\}\]
is exponentially distributed with parameter $\nu(B_{\e^{-\rho}}^c(0))$ and that 
it is independent of the process of $(Z_t)_{0\lqq t < \tau}$ and hence $(X^\e_{t})_{0\lqq t< \tau}$.  
Since by the regular variation of $\nu$ we have $\nu(B_{\e^{-\rho}}^c(0)) / \e^{-\al \rho} \mu(B_1^c(0)) \ra 1$ as $\e\ra 0$, 
there exists a constant $\rho_0 \in (0,1)$ such that for any $\rho \in (0, \rho_0]$ 
\[
\PP\big(\tau> \frac{2r_\e}{h_\e}\big) = \exp\big(-\frac{2r_\e \nu(B_{\e^{-\rho}}^c(0))}{h_\e}\big) \to 1 \quad \mbox{ as }\e \to 0.
\]
In particular, we may choose the upper bounds $\rho_0, \e_0\in (0,1)$ such that for all $\e \in (0, \e_0]$ and $\rho \in (0, \rho_0]$ 
we have $1-\exp(-2r_\e \nu(B_{\e^{-\rho}}^c(0))/h_\e) < \gamma/3\|\psi\|_\infty$. 
For convenience we denote by $\ti \PP$ and $\ti \EE$ the probability measure $\PP(\,\cdot\, |\tau> 2r_\e/h_\e)$ 
and its expectation.
We may assume without loss of generality that $\psi$ is uniformly continuous on $\RR^d$, 
we denote its modulus of continuity by $\varpi_\psi$. Since $\varphi_\psi(\beta) \ra 0$ as $\beta\ra 0$, we may choose $\beta_0\in (0,1)$ 
such that for all $\beta \in (0, \beta_0]$ we have $\varpi(\beta) \lqq \gamma/3$.  
For fixed $\beta \in (0, \beta_0]$ we apply Corollary 3.1 in \cite{HoePav-14} for the upper bound $2\tT^*$ of $\tT_\e$, 
which provides the existence of constants $p_0, \e_0 \in (0,1)$ such that for all $p\in (0, p_0]$ and $\e\in (0, \e_0]$ 
\begin{equation}
\label{eq: main step}
\begin{aligned}
&\ti \EE_y\Big[\frac{h_\e}{2r_\e} \int_0^{\ell_\e \tT_\e} \psi(X^\e_s)\, ds\Big]
= \frac{h_\e}{2r_\e}
\ti \EE_y\Big[\int_{0}^{\tT_\e}\psi(X^\e_s)\, ds
+ \int_{\tT_\e}^{\ell_\e \tT_\e}\psi(X^\e_s)\, ds\Big]\\
&\lqq \frac{h_\e}{2r_\e}
\ti \EE_y\Big[
\Big(\int_{0}^{\tT_\e}\psi(X^\e_s)\, ds + \int_{\tT_\e}^{\ell_\e \tT_\e}\psi(X^\e_s)\, ds\Big) 
\ind\Big\{\sup_{s\in [0, \tT_\e]} \|X^\e_s - \phi_s(y)\| < \beta\Big\}
\Big] + \|\psi\|_\infty e^{-\e^{-p}}
\end{aligned}
\end{equation}
We continue with the first term. Recall that by construction $\ell_\e \tT_\e = 2r_\e/h_\e$. 
We are now in the position to apply the Markov property of $X^\e$ again and obtain the recursion 
\begin{align*}
&\frac{h_\e}{2r_\e}
\ti \EE_y\Big[
\Big(\int_{0}^{\tT_\e}\psi(X^\e_s)\, ds + \int_{\tT_\e}^{\ell_\e \tT_\e}\psi(X^\e_s)\, ds\Big) 
\ind\Big\{\sup_{s\in [0, \tT_\e]} \|X^\e_s - \phi_s(y)\| < \beta\Big\}
\Big] \\
&= \frac{h_\e \tT_\e}{2r_\e} 
\ti \EE_y\Big[
\frac{1}{\tT_\e} 
\int_{0}^{\tT_\e}\psi(X^\e_s)\, ds \cdot
\ind\Big\{\sup_{s\in [0, \tT_\e]} \|X^\e_s - \phi_s(y)\| < \beta\Big\}
\Big]
\\
&\qquad + 
\frac{h_\e}{2r_\e} 
\ti \EE_y\Big[\ti \EE_{X^{\e}_{\tT_\e}}\Big[\int_{0}^{(\ell_\e-1) \tT_\e}\psi(X^\e_s)\, ds\Big] \cdot
\ind\{X^\e_{\tT_\e} \in \ti D^\iota\}\Big]\\
&\lqq \frac{h_\e \tT_\e}{2r_\e} 
\Big(\frac{1}{\tT_\e} 
\int_{0}^{\tT_\e}\psi(\phi_s(y))\, ds + \varpi_{\psi}(\beta)\Big)
+ \frac{h_\e}{2r_\e} 
\sup_{z\in \ti D^\iota} 
\ti \EE_{z}\Big[\int_{0}^{(\ell_\e-1)\tT_\e}\psi(X^\e_s)\, ds\Big] \\
&\lqq \frac{1}{\ell_\e} 
\Big(
\int_{K^\iota}\psi(v)\, dP^\iota(dv) + \frac{\gamma}{3} \Big)
+  \sup_{z\in \ti D^\iota} 
\ti \EE_{z}\Big[\frac{h_\e}{2r_\e}\int_{0}^{(\ell_\e-1)\tT_\e}\psi(X^\e_s)\, ds\Big]. 
\end{align*}
Iterating the step in \eqref{eq: main step} $\ell_\e-1$ times and choosing $\e_0 \in (0, 1)$ such that 
for all $\e \in (0, \e_0]$ we have $\ell_\e \|\psi\|_\infty \exp(-\e^{-p})< \gamma/3$ we obtain  
\begin{align*}
\ti \EE_y\Big[\frac{h_\e}{2r_\e} \int_0^{\ell_\e \tT_\e} \psi(X^\e_s)\, ds\Big]
&\lqq \Big(
\int_{K^\iota}\psi(v)\, dP^\iota(dv) + \frac{\gamma}{3} \Big) + \frac{\gamma}{3},
\end{align*}
and eventually end up with 
\begin{align*}
\EE_y\Big[\frac{h_\e}{2r_\e} \int_0^{\frac{2r_\e}{h_\e}} \psi(X^\e_s)\, ds\Big] 
&\lqq \ti \EE_y\Big[\frac{h_\e}{2r_\e} \int_0^{\frac{2r_\e}{h_\e}} \psi(X^\e_s)\, ds\Big] 
+ \frac{\gamma}{3}\\
&\lqq \int_{K^\iota}\psi(v)\, dP^\iota(dv) + \gamma.
\end{align*}
The lower estimate follows analogously. This finishes the proof. 
\end{proof}

\begin{lem}\label{lem: 2}
Let $\delta^*$, $\delta'$, $\delta>0$ and $R>R_0$ be chosen as above. 
If $0< \gamma' < \delta_0$ and $\iota \in \{1, \dots, \kappa\}$ 
then there is a constant $\e_0 \in (0,1)$ such that for any $x\in D^\iota_{\delta, \delta'}$ and $\e\in (0, \e_0]$  
\begin{equation}\label{eq: approximation on time scale}
\PP_x\Big(X^\e_{\frac{t-r_\e}{h_\e}} \in D^{\iota, R}_{\delta, \delta'}\Big) 
\lqq (1+ \gamma')\PP_x\Big(X^\e_{\frac{t}{h_\e}} \in D^{\iota, R}_{\delta, \delta'}\Big).
\end{equation}
\end{lem}

\begin{proof} Fix $0< \gamma' < \delta_0$. For convenience we return to the abbreviation $\ti D^\iota$.
With the help of the Markov property we obtain 
\begin{align*}
\PP_x\Big(X^\e_{\frac{t}{h_\e}} \in \ti D^\iota\Big) 
&= \sum_{\ell =1}^\kappa \PP_x\Big(X^\e_{\frac{t-r_\e}{h_\e}} \in \ti D^\ell\Big) 
\PP\Big(X^\e_{\frac{t}{h_\e}} \in \ti D^\iota\Big| X^\e_{\frac{t-r_\e}{h_\e}} \in \ti D^\ell\Big)\\
&\lqq \sum_{\ell =1}^\kappa \PP_x\Big(X^\e_{\frac{t-r_\e}{h_\e}} \in \ti D^\ell\Big) 
\sup_{y\in \ti D^\ell} \PP_y\Big(X^\e_{\frac{r_\e}{h_\e}} \in \ti D^\iota\Big)\\
&\lqq \sum_{\ell =1}^\kappa \PP_x\Big(X^\e_{\frac{t-r_\e}{h_\e}} \in\ti D^\ell\Big) 
\sup_{y\in \ti D^\ell} \PP_y\Big(X^\e_{\frac{r_\e}{h_\e}} \in \ti D^\iota\Big).
\end{align*}
For $\ell \neq \iota$, the first exit time $\TT^{\iota, R}_x(\e)$ satisfies the following estimate. 
For any $C\in (0,1)$ there is a constant $\e_0 \in (0,1)$ such that for $\e \in (0, \e_0]$ 
\begin{align*}
\sup_{y\in \ti D^\ell} \PP_y\Big(X^\e_{\frac{r_\e}{h_\e}} \in \ti D^\iota\Big) 
&= \sup_{y\in \ti D^\ell} \PP_y\Big(\TT^{\iota, R}(\e) \lqq \frac{r_\e}{h_\e}\Big) \\
&= \sup_{y\in \ti D^\ell} \PP_y\Big(Q((\ha D^\ell)^c) h_\e \TT^{\iota, R}(\e) 
\lqq Q((\ha D^\ell)^c) r_\e\Big) \\
&\lqq (1+C)\big(1-e^{Q((\ha D^\ell)^c) r_\e}\big).
\end{align*}
The last estimate in the preceding formula
is a direct consequence of the convergence result in Corollary 2.1 of \cite{HoePav-14}. 
Reducing $\e_0$ further if necessary we obtain 
$(1+C)(1-\exp(Q((\ha D^\ell)^c) r_\e)) \lqq \gamma'/\kappa-1$ for $\e \in (0, \e_0]$
and the desired result holds, namely
\begin{align*}
\PP_x(X^\e_{\frac{t}{h_\e}} \in \ti D^\iota) 
&\lqq \PP_x(X^\e_{\frac{t-r_\e}{h_\e}} \in \ti D^\ell)(1+\gamma').
\end{align*}
\end{proof}

\paragraph{Conclusion of the Proof of Theorem \ref{thm: first exit times}:} 
We apply the Lemmas \ref{lem: 1} and \ref{lem: 2} with the choices $\gamma = \gamma' = \delta^*$, 
as well as the minimal value of all $\e_0$ to the right-hand side of inequality \eqref{eq: Markov estimate} and obtain 
for $\e\in (0, \e_0]$ 
\begin{align*}
\EE_x\Big[ \psi(X^\e_{\frac{t+\si r_\e}{h_\e}})\Big]
&\lqq \sum_{\iota =1}^\kappa \Big(\int_{K^\iota} \psi(v) \, dP^\iota(v) + \delta^*\Big)\;
(1+\delta^*) \PP_x\Big(X^\e_{\frac{t}{h_\e}} \in \ti D^\iota\Big) + \|\psi\|_\infty \delta^*\\
&\lqq \sum_{\iota =1}^\kappa \Big(\int_{K^\iota} \psi(v) \, dP^\iota(v)\Big)\;
\PP_x\Big(X^\e_{\frac{t}{h_\e}} \in \ti D^\iota\Big) + \delta^*(1+\delta^*)+ \|\psi\|_\infty \delta^*\\
&= \EE_\iota\Big[\int_{\RR^d} \psi(v) \, dP^{m^\delta_{t}}(v)\Big]+ \delta^*\Big((1+\delta^*)+ \|\psi\|_\infty\Big).
\end{align*}
With the analogous arguments we obtain  
\begin{align*}
\EE_x\Big[ \psi(X^\e_{\frac{t+\si r_\e}{h_\e},x})\Big]  
&\gqq \EE_\iota\Big[\int_{\RR^d} \psi(v)\, dP^{m^\delta_{t}}(v)\Big] -\delta^*\Big(1+\delta^* + \|\psi\|_\infty\Big).
\end{align*}
This finishes the proof. 

\section{Appendix} 
\subsection{Proof of Lemma \ref{lem: reduced domains}}

We fix the maximal distance $\delta_0 := \frac12\min_{\iota}  \dist(K^\iota, \ds\cup_{\iota=1}^\kappa \pd D^\iota)$, 
the minimal cutoff for the domain $\ds R_0 := \inf\{r>0\colon  \cup_{\iota=1}^\kappa K^\iota \subset B_r(0)\}$ and 
an index $\iota \in \{1, \dots, \kappa\}$. 
\begin{enumerate}
 \item Fix $0 < \delta < \delta_0$ and $R>R_0$. 
 Claim: We have $\phi_t(D^{\iota, R}_{\delta}) \subset D^{\iota, R}_\delta$ for all $t\gqq 0$. \\
We use that $\phi_{-t} = \phi_t^{-1}$, the intersection compatibility of preimages, the definition of $D^{\iota, R}_{\delta}$, 
as well as iterated De Morgan's rules to obtain 
\begin{align}
D^{\iota, R}_\delta &= (D^\iota \cap \iI_R) \setminus \bigcup_{t\gqq 0} \phi_{-t}\Big(B_\delta(\pd D^\iota) \cap D^\iota \cap \iI_{R}\Big)\nonumber\\
&= (D^\iota \cap \iI_R) \cap \bigcap_{t \gqq 0} \phi_{-t}\big((D^\iota \cap \iI_{R}) \setminus (B_\delta(\pd D^\iota) \cap D^\iota \cap \iI_{R})\big). 
\label{eq: representation of the r.d.o.a.}
\end{align}
Using the positive invariance of $\iI_R$ and the injectivity of the flow $x \mapsto \phi_t(x)$ for all $x\in \RR^d$ we obtain for $s\gqq 0$ that 
\begin{align*}
\phi_s(D^{\iota, R}_\delta) 
&= \phi_s (D^\iota)  \cap \phi_s(\iI_R) \cap \bigcap_{t \gqq 0} 
\phi_{s}\Big(\phi_{-t}\big((D^\iota \cap \iI_{R}) \setminus (B_\delta(\pd D^\iota) \cap D^\iota \cap \iI_{R})\big)\Big) \\
&=\ D^\iota \cap \phi_s(\iI_R) \cap \bigcap_{t \gqq 0} \phi_{s}
\Big(\phi_{-t}\big((D^\iota \cap \iI_{R}) \setminus (B_\delta(\pd D^\iota) \cap D^\iota \cap \iI_{R})\big)\Big) \\
&= D^\iota \cap \phi_s(\iI_R) \cap \bigcap_{t \gqq 0} \phi_{-t}
\Big((D^\iota \cap \iI_{R}) \setminus (B_\delta(\pd D^\iota) \cap D^\iota \cap \iI_{R})\Big) \\
&\qquad \cap \bigcup_{0< t \lqq s} \Big((D^\iota \cap \iI_{R}) \setminus (B_\delta(\pd D^\iota) \cap D^\iota \cap \iI_{R})\Big) \\
&\subset D^\iota \cap \iI_R \cap \bigcap_{t \gqq 0} \phi_{-t}
\Big((D^\iota \cap \iI_{R}) \setminus (B_\delta(\pd D^\iota) \cap D^\iota \cap \iI_{R})\Big) = D^{\iota, R}_\delta.
\end{align*}

\item 
Fix $0 < \delta < \delta_0$, $R>R_0$ and in addition $0< \gamma < \delta_0$. Claim: 
 there is a constant $T^* = T^*_{\delta, R, \gamma}>0$ 
such that for all $x\in D^{\iota, R}_{\delta}$ and $t\gqq T^*$
\[
u(t;x) \in B_{\gamma}(K^\iota).
\]
Since $K^\iota$ is an attractor, it attracts all bounded closed sets in its domain of attraction. 
$\overline D^{\iota, R}_\delta$ is bounded closed set in $D^\iota$. 
That means for any $\gamma>0$ there is $T^* = T^*(\gamma)$ such that for all $t\gqq T^*$ 
\[
\phi_t\Big(D^{\iota, R}_\delta\Big) \subset \bB_{\gamma}(K^\iota).
\]

\item 
Claim: If $0< \delta < \delta'< \delta_0$ and $R>R_0$, then $D^{\iota, R}_{\delta'} \subset D^{\iota, R}_{\delta}$.\\
This follows immediately from the representation \eqref{eq: representation of the r.d.o.a.} 
by the monotonicity with respect to inclusion of $\delta$, which is stable under preimages.

\item 
Claim: If $\delta, \delta'>0$ such that $\delta+ \delta'<\delta_0$ and $R>R_0$ , then 
$\phi_t(D^{\iota, R}_{\delta, \delta'}) \subset D^{\iota, R}_{\delta, \delta'}$ for all $t\gqq 0$. \\
The proof is virtually identical to the proof of 1, with $D^\iota \cap \iI_R$ replaced by $D^{\iota, R}_{\delta}$. 

\item 
Claim: If $\delta, \delta', \delta''>0$ with $\delta'< \delta''$ and $\delta + \delta''< \delta_0$, then 
$D^{\iota, R}_{\delta, \delta'} \subset D^{\iota, R}_{\delta, \delta''}$.\\
This follows analogously to Claim 3.

\item 
Claim: We have 
\[
\bigcup_{\substack{\delta, \delta'>0\\ \delta+\delta'<\delta_0}} D^{\iota, R}_{\delta, \delta'} = D^{\iota} \cap \iI_R.
\]
We first prove that 
\[
\bigcup_{\substack{0 < \delta < \delta_0}} D^{\iota, R}_{\delta} = D^{\iota} \cap \iI_R.
\]
Recall that by Claim 3 
the family $\big(D^{\iota, R}_{\delta}\big)_{\delta>0}$ is monotically decreasing 
as a function of $\delta$ with respect to the set inclusion. 
For any $x\in D^{\iota} \cap \iI_R$, it is sufficient to find 
$\delta>0$ such that 
\[
x \in \bigcap_{t \gqq 0} \phi_{-t}\big((D^\iota \cap \iI_{R}) \setminus (B_\delta(\pd D^\iota) \cap D^\iota \cap \iI_{R})\big)
\]
Assume $\delta>0$ such that in addition $x\in (D^\iota \cap \iI_{R}) \setminus B_\delta(\pd D^\iota)$. 
Then due to the continuity of $t \mapsto \phi_t(x)$, 
there is $T_\delta = T_\delta(x)>0$ such that 
\[
x \in \bigcap_{0\lqq t< T_\delta}  \phi_{-t}\big((D^\iota \cap \iI_{R}) \setminus (B_\delta(\pd D^\iota) \cap D^\iota \cap \iI_{R})\big).
\]
Furthermore, $\delta \mapsto T_\delta$ is monotonically decreasing and continuous. 
We prove that $\lim_{\delta \ra 0+} T_\delta = \infty$. 
Assume $T_\infty := \sup_{\delta>0} T_\delta < \infty$, then for any $\delta>0$  
\[
\phi_{-(T_{\infty} +1)}(x) \in D^\iota \cap B_\delta(\pd D^\iota) 
\]
and hence 
\[
\phi_{-(T_{\infty} +1)}(x) \in \bigcap_{\delta>0} D^\iota \cap B_\delta(\pd D^\iota) = \pd D^\iota,
\]
which is a contradiction, since $\phi_t(D^\iota) = D^\iota$ for all $t\in \RR$.
Hence $T_\infty = \infty$ and we find $\delta>0$ such that $x\in D^{\iota, R}_{\delta}$. 
The same reasoning holds analogously for $D^{\iota, R}_{\delta}$ replaced by $D^{\iota, R}_{\delta, \delta'}$ 
and $D^{\iota} \cap \iI_R$ by $D^{\iota, R}_{\delta}$.  
\end{enumerate}

\subsection{Local Morse--Smale flows satisfy the local ergodicity property}

It suffice to prove the convergence result for a stable limit cycle $K$ and its domain of attraction $D$. 

\begin{lem}
Consider a stable limit cycle $K$ and its domain of attraction $D$. Denote by $\tT$ the period of $\phi$ on $K$ and $x_0 \in K$. 
Then for any compact subset $A \subset D$ and measurable set $B \in \bB(\RR^d)$ the limit 
\begin{align*}
\lim_{T\ra \infty} \sup_{x \in A} \Big|\frac{1}{T} \int_0^T \ind_B(\phi_s(x))\, ds 
- \frac{1}{\tT} \int_0^{\tT} \ind_B(\phi_s(x_0))\, ds\Big|
\end{align*}
holds true.
\end{lem}

\noindent
\textit{Sketch of the proof}.
First of all note that due to the compactness of $A$ and the openness of $D$ there is a 
minimal positive distance between $A$ and $\pd D$. 
Since $K$ is a global attractor in $D$, for any $\delta>0$ there is $T_{\delta, A}>0$ such that $x\in A$ and 
$t\gqq T_{\delta, A}$ implies 
\[
\phi_t(x) \in \bB_{\delta}(K). 
\]
It is therefore sufficient to prove that
\begin{align*}
\lim_{T\ra \infty} \sup_{x \in \bB_\delta(K)} \Big|\frac{1}{T} \int_0^T \ind_B(\phi_s(x))\, ds 
- \frac{1}{\tT} \int_0^{\tT} \ind_B(\phi_s(x_0))\, ds\Big|.
\end{align*}
Note further that the value $\frac{1}{\tT} \int_0^{\tT} \ind_B(\phi_s(x_0))\, ds$ 
is independent of $x_0\in K$ and trivially 
\[
\frac{1}{\tT} \int_0^{\tT} \ind_B(\phi_s(x_0))\, ds = \frac{1}{n \tT} \int_0^{n\tT} \ind_B(\phi_s(x_0))\, ds.
\]
It is sufficient to check the case $T_n = n \tT$. In this case it is therefore enough to show 
\begin{align*}
\lim_{n\ra \infty} \sup_{x \in \bB_\delta(K)} \Big|\frac{1}{n \tT} \int_0^{n \tT} \ind_B(\phi_s(x))\, ds 
- \frac{1}{n\tT} \int_0^{n\tT} \ind_B(\phi_s(x_0))\, ds\Big|.
\end{align*}
We calculate for $x\in \bB_{\delta}(K)$ and $n\in \NN$ 
\begin{align*}
&\frac{1}{n \tT} \int_0^{n \tT} \ind_B(\phi_s(x)) ds - \frac{1}{n\tT} \int_0^{n\tT} \ind_B(\phi_s(x_0))\, ds \\
&= \frac{1}{n\tT} \sum_{i=1}^n  \int_{(i-1) \tT}^{i \tT} \big(\ind_B(\phi_s(x)) -\ind_B(\Pi_K(\phi_s(x)))\big)\, ds,
\end{align*}
where $\Pi_K$ is the (local) orthogonal projection of $x \in \bB_{\delta}(K)$ onto the smooth curve $K$. 
The hyperbolicity of $K$ and the compactness of $K$ imply that for $\delta>0$ sufficiently small, there exist 
a constant $C_\delta$ and $\la>0$ such that the sequence 
\begin{align*}
f_n:= \sup_{x\in K} \sup_{s\in [(n-1)\tT, n\tT]} |\phi_s(x) - \Pi_K \phi_s(x)|, \quad n\in \NN, 
\end{align*}
satisfies $f_n \lqq C_\delta e^{-\la n}\ra 0$ for all $n\in \NN$.  
This uniform convergence implies the convergence of the Lebesgue integral
\[
\int_{(n-1) \tT}^{n \tT} \big(\ind_B(\phi_s(x)) -\ind_B(\Pi_K(\phi_s(x)))\big)\, ds \ra 0, \quad \mbox{ as }n\ra \infty,
\]
and hence the desired convergence 
\begin{align*}
&\frac{1}{n \tT} \int_0^{n \tT} \ind_B(\phi_s(x))\, ds 
\ra \frac{1}{\tT} \int_0^{\tT} \ind_B(\phi_s(x_0))\, ds \quad \mbox{ as }n\ra \infty.
\end{align*}
\hfill $\square$

\section*{Acknowledgements}
The first author expresses his gratitude to the Berlin Mathematical School (BMS), the 
International Research Training Group (IRTG) 1740: ``Dynamical Phenomena in Complex Networks: 
Fundamentals and Applications'' and the Chair of Probability theory of Universit\"at Potsdam 
for various infrastructure support.

\end{document}